\documentclass[a4paper,11pt]{article}
\usepackage[T1]{fontenc}
\usepackage{lmodern}
\usepackage{amsmath,amssymb,amsthm,mathrsfs}
\usepackage{ifxetex}
\ifxetex
  \PassOptionsToPackage{xetex}{geometry}
\fi
\usepackage[margin=27mm]{geometry}
\usepackage[hyperfootnotes=false]{hyperref}
\hypersetup{unicode=true,colorlinks=true,linkcolor=blue,citecolor=blue,urlcolor=blue}
\allowdisplaybreaks[2]
\numberwithin{equation}{section}
\theoremstyle{plain}
\newtheorem{theorem}{Theorem}[section]
\newtheorem{lemma}[theorem]{Lemma}
\newtheorem{proposition}[theorem]{Proposition}
\theoremstyle{definition}
\newtheorem{remark}[theorem]{Remark}
\newtheorem*{xremark}{Remark}
\newtheorem*{xnotation}{Notation}
\newcommand{\Q}{\mathbb Q}
\newcommand{\R}{\mathbb R}
\newcommand{\Z}{\mathbb Z}

\newcommand{\PP}{\mathbb P}
\newcommand{\FF}{\mathbb F}
\newcommand{\bx}{\mathbf x}
\newcommand{\by}{\mathbf y}
\newcommand{\bn}{\mathbf n}
\newcommand{\bd}{\mathbf d}

\newcommand{\bL}{\mathbf L}
\newcommand{\br}{\mathbf r}
\newcommand{\bt}{\mathbf t}
\newcommand{\one}{\mathbf 1}
\newcommand{\SSS}{\mathfrak S}
\newcommand{\cP}{c_{\mathrm P}}
\newcommand{\vol}{\operatorname{vol}}
\newcommand{\supp}{\operatorname{supp}}
\newcommand{\Pic}{\operatorname{Pic}}

\newcommand{\prim}{\mathrm{prim}}
\newcommand{\sfpart}{\mathrm{sf}}
\newcommand{\eps}{\varepsilon}

\title{Counting Rational Points on a Certain Bihomogeneous Threefold}
\author{Xiaodong Zhao}
\date{}

\begin{document}
\pagestyle{plain}
\maketitle

\begingroup
\renewcommand{\thefootnote}{}
\makeatletter
\renewcommand{\@makefntext}[1]{\noindent#1}
\makeatother
\footnotetext{%
\textit{2020 Mathematics Subject Classification.}
11D45; 11G50, 14G05.\par
\noindent\textit{Keywords.} Manin--Peyre conjecture; bihomogeneous threefold;
primitive lattice points; quadratic character sums.}
\endgroup

\begin{abstract}
Let $V\subset\mathbb P^2\times\mathbb P^2$ be the threefold
$x_0y_0^2+x_1y_1^2+x_2y_2^2=0$, and let $U$ be the complement
of its coordinate hyperplane sections. For the anticanonical height
$H([\mathbf x],[\mathbf y])=|\mathbf x|^2|\mathbf y|$, evaluated on
primitive integral representatives, we prove
\[
 N_U(B)=\frac{12\pi-14}{\zeta(2)\zeta(3)}B\log B+o(B\log B).
\]
The two projections combine lattice counting with uniform conic
estimates. The main arithmetic step is an average of the fiber
singular series, whose two main terms are related by Hilbert
reciprocity. A moment estimate controls the full bad-prime part,
and shape estimates permit the errors to be summed over the height
region. We also check the leading constant with Peyre's constant.
\end{abstract}

\section{Introduction}\label{sec:intro}

Counting rational points of bounded height is a quantitative form of
Diophantine geometry: instead of asking only whether rational points
exist, one asks how their distribution reflects the geometry and
arithmetic of the underlying variety. The program initiated by
Franke, Manin, and Tschinkel \cite{FMT} relates this distribution on
Fano varieties to the anticanonical class and the Picard group.
In the modern formulation, assume that $Z$ is a smooth Fano variety
over $\Q$ and that $Z(\Q)$ is not thin. For an anticanonical height
$H_Z$ and Picard rank $\rho(Z)$, the expected asymptotic on the
complement $\Omega\subset Z(\Q)$ of a suitable thin exceptional set
has the form
\[
\#\{P\in\Omega:H_Z(P)\le B\}
\sim c_{Z,H_Z}B(\log B)^{\rho(Z)-1}.
\]
Peyre \cite{Peyre} proposed a geometric and adelic interpretation
of the leading constant.

The choice of $\Omega$ is an essential part of the question.
Batyrev and Tschinkel \cite{BT96} exhibited Fano cubic bundles over
suitable number fields for
which removing only a proper closed subset does not restore the
original prediction. Modern formulations allow suitable thin
exceptional sets; these can include images of rational points under
dominant generically finite maps of degree greater than one.
Lehmann, Sengupta, and Tanimoto \cite{LST} established geometric
consistency results for such exceptional sets.

In this paper, we consider the smooth bihomogeneous threefold
\[
V:\quad F(\bx,\by):=x_0y_0^2+x_1y_1^2+x_2y_2^2=0
\quad\subset\PP^2_{\bx}\times\PP^2_{\by},
\]
and the open subset
\[
U=V\setminus\{x_0x_1x_2y_0y_1y_2=0\}.
\]
For a nonzero vector $\bx=(x_0,x_1,x_2)\in\Q^3$, the notation
$
[\bx]=[x_0:x_1:x_2]\in\PP^2(\Q)$
denotes its projective class: two nonzero vectors represent the same
class precisely when one is a nonzero rational multiple of the other.
The notation $[\by]$ has the same meaning.
Choose primitive integral representatives
$\bx,\by\in\Z_{\prim}^3$. We take the anticanonical height
$H:V(\Q)\to\R_{\ge0}$ to be
$
H([\bx],[\by])=|\bx|^2|\by|,$
where $|\cdot|:\R^3\to\R_{\ge0}$ is the sup-norm.
The primitive integral representatives are unique up to sign, so both
this height and the expressions $|\bx|$ and $|\by|$ are well defined
on the corresponding projective points.
Write
\[
N_U(B)=
\#\left\{
([\bx],[\by])\in U(\Q):
H([\bx],[\by])\le B
\right\}.
\]

For the threefold $V$ considered here, adjunction gives
$-K_V=2H_x+H_y$, where $H_x,H_y$ are the restrictions of the two
hyperplane classes, and $\rho(V)=2$; see
Proposition~\ref{prop:peyre}. This explains both the height
$|\bx|^2|\by|$ and the expected factor $B\log B$.
The two projections reveal different counting problems.
Fixing $\by$ gives a projective line, whereas fixing $\bx$ with
$x_0x_1x_2\ne0$ gives a nonsingular diagonal conic. The resulting
combination of linear geometry and quadratic arithmetic makes this
a natural low-dimensional test of the Manin--Peyre prediction.
For this specific $V$ and the open subset $U$ used here, Le Boudec
\cite[Theorem~1]{LB} proved the sharp order bounds
\begin{equation}\label{eq:LBupper}
B\log B\ll N_U(B)\ll B\log B.
\end{equation}

In 2020, Browning and Heath-Brown \cite{BH2020} established
the Manin--Peyre conjecture for
\[ x_0y_0^2+x_1y_1^2+x_2y_2^2+x_3y_3^2=0
\quad\subset\PP^3_{\bx}\times\PP^3_{\by},
\]
after removing an appropriate thin set of rational points.
For the analogous diagonal hypersurface in
$\PP^{s-1}\times\PP^{s-1}$, Wang \cite{W1} established
an asymptotic formula for $s\ge7$.
The author \cite{Z1} treated the corresponding family for $n\ge5$.

In 2024, Bonolis, Browning, and Huang \cite{BBH2024} proved the Manin--Peyre conjecture for certain smooth threefolds
\[
\sum_{i=0}^3 L_i(s,t)v_i^2=0
\quad\subset\PP^1_{(s:t)}\times\PP^3_{\mathbf v}.
\]
Their theorem assumes that the integral binary linear forms $L_i$
are primitive and pairwise non-proportional and that the common
intersection of the two quadrics in the associated pencil has no
real points. With the sup-norm height $|(s,t)|\,|\mathbf v|^2$, and after removing
the square-discriminant fibers, the count is asymptotic to
Peyre's constant times $B\log B$. This is progress in dimension
three, but for quadric-surface bundles over $\PP^1$, not conic bundles
over $\PP^2$. This distinction matters for $V$.

A closely related result of Heath-Brown and Loughran
\cite[Theorem~1.2]{HL2025} concerns smooth surfaces
\[
S:\quad sQ_0(\mathbf z)+tQ_1(\mathbf z)=0
\subset \PP^1_{(s:t)}\times\PP^2_{\mathbf z},
\]
where $Q_0,Q_1$ are ternary quadratic forms.
If their common zero locus has no rational point and every
nonzero rational linear combination of $Q_0,Q_1$ is nonsingular,
they prove the Manin--Peyre asymptotic for the complement of
the $(-1)$-curves in $S$.
These quintic del Pezzo surfaces provide a lower-dimensional
conic-bundle counterpart to the quadric-surface bundles above.
Their results do not directly cover $V$, whose conic-bundle
base is $\PP^2$.

Recently, Browning, Munshi, and Wang \cite{BMW} obtained
a $T^3\log T$ main term for integer solutions of $F(\bx,\by)=0$
scaled by one parameter $T$, with a fixed nonnegative smooth compactly
supported weight in $y_0y_1y_2\ne0$. This concerns a cubic fourfold
in $\PP^5$, not the separately primitive biprojective count with
$|\bx|^2|\by|\le B$; see \cite[Remark~1.5]{BMW}.

The threefold $V$ also belongs to the broader class of biprojective
hypersurfaces. Schindler \cite{Schindler} developed
circle-method estimates in boxes with different side lengths and
deduced Manin-type asymptotics for biprojective hypersurfaces under
degree and dimension hypotheses. In particular, her main theorem
assumes that both components of the bidegree are at least two, and
does not cover the present threefold of bidegree $(1,2)$.
Browning and Hu \cite{B5} obtained improved results for biquadratic
hypersurfaces. In the multiprojective setting, Blomer and
Br\"{u}dern \cite{B6} established asymptotic formulae for diagonal
equations, and the author \cite{Z2} proved the Manin--Peyre conjecture
for certain multiprojective hypersurfaces.

\begin{theorem}\label{thm:main}
For the open subset $U$ and height $H$ defined above, as $B\to\infty$, we have
\begin{equation}\label{eq:main}
N_U(B)=\cP B\log B+o(B\log B),
\end{equation}
where
\[
\cP=\frac{12\pi-14}{\zeta(2)\zeta(3)},
\]
which equals Peyre's constant.
\end{theorem}

The two projections require different estimates. Fixing $\mathbf y$
gives a rank-two lattice in the $\mathbf x$-variables; fixing a
coefficient vector with nonzero coordinates gives a nonsingular conic.
We separate these methods at $|\mathbf y|=B^{1/3}$.
\begin{theorem}\label{thm:split}
Write
\[
N_{\mathrm{low}}(B)=
\#\{([\bx],[\by])\in U(\Q):
H([\bx],[\by])\le B,\ |\by|\le B^{1/3}\},
\]
\[
R(B)=
\#\{([\bx],[\by])\in U(\Q):
H([\bx],[\by])\le B,\ |\by|>B^{1/3}\}.
\]
With $\cP$ as in Theorem~\ref{thm:main}, we have
\begin{align}
N_{\mathrm{low}}(B)
&=\frac{\cP}{3}B\log B+O(B\log\log B),\label{eq:low-main}\\
R(B)&=\frac{2\cP}{3}B\log B+o(B\log B).\label{eq:high-main}
\end{align}
\end{theorem}

Theorem~\ref{thm:main} follows from Theorem~\ref{thm:split} and the constant
calculation in Section~\ref{sec:constant}.
The proof separates the height region at $|\by|=B^{1/3}$.
The low-$y$ contribution is treated by primitive lattice counting,
while the complementary region requires an average of conic leading
coefficients and a summable error estimate.
Abel summation then converts these latter estimates into the
high-$y$ asymptotic.

The unequal proportions $1/3$ and $2/3$ also admit an
effective-cone interpretation, described after
Proposition~\ref{prop:peyre}.
An analogous phenomenon occurs in
\cite[(1.3) and \S~2.2]{HL2025}, where the two regions
contribute proportions $1/4$ and $3/4$.
This geometric comparison explains the constants but does
not replace the analytic proof of Theorem~\ref{thm:split}.

Section~\ref{sec:inputs} states the external inputs.
Sections~\ref{sec:constant} and~\ref{sec:low} evaluate the Peyre constant
and the low-$y$ contribution.
Sections~\ref{sec:sf}--\ref{sec:arithmean} establish the singular-series
average. Section~\ref{sec:shape} proves the shape and interval bounds,
and Section~\ref{sec:K} obtains the mean of the fiber coefficients.
Sections~\ref{sec:error} and~\ref{sec:conclusion} complete the error
summation and the Abel summation.
The low-$y$ argument uses lattice counting together with
\eqref{eq:LBupper}, whereas the shape bounds use
\cite[Lemmas~4 and~5]{LB}. Both arguments are independent of
the singular-series average.

\begin{xnotation}
The notation $\Z_{\prim}^3$ denotes nonzero primitive integral vectors.
We write $p\nmid\bx$ when the three coordinates are not all divisible
by $p$, and $\supp(n)$ for the prime divisors of a positive integer $n$.
The function $\tau_k(n)$ is the $k$-fold divisor function,
$\tau=\tau_2$, and $\omega(n)$ is the number of distinct prime divisors.
The symbol $(\frac{a}{d})$ denotes the Jacobi symbol for odd positive $d$,
with value $1$ when $d=1$. Also, $u_+=\max(u,0)$.
All $p$-adic Haar measures give $\Z_p^3$ volume $1$. For any vector \(\mathbf v=(v_0,v_1,\ldots,v_n)\in\mathbb R^{n+1}\), we write
$
\|\mathbf v\|_2:=\sqrt{v_0^2+v_1^2+\cdots+v_n^2}$
for its Euclidean norm, whereas $
|\mathbf v|:=\max_{0\le i\le n}|v_i|$
denotes its sup-norm.
Coefficient vectors in projective sums are always represented by
primitive integral vectors; expressions involving absolute values do
not depend on the choice of sign.

A sum written $\sum_{|\bx|\le t}a_{[\bx]}$ runs over projective
coefficient points $[\bx]\in\PP^2(\Q)$ with $x_0x_1x_2\ne0$,
where $\bx$ is a primitive integral representative and $|\bx|\le t$.
Each projective point is counted once; opposite primitive vectors are
not counted separately. Additional restrictions below such a sum
restrict the same set of projective points.
Dyadic sums are over powers of two, starting at one.
Subscripts on $O$ and $\ll$ list the allowed dependence on fixed
parameters. In each limiting argument, these parameters are fixed
until the limit in the height has been taken.
\end{xnotation}

\section{Local densities and external counting inputs}\label{sec:inputs}

For a primitive coefficient vector with $x_0x_1x_2\ne0$, put
$q_{\bx}(\by)=x_0y_0^2+x_1y_1^2+x_2y_2^2,$ we have $\nabla q_{\bx}(\by)=(2x_0y_0,2x_1y_1,2x_2y_2)$.
Since $x_0x_1x_2\ne0$, $\nabla q_{\bx}(\by)$ vanishes only at
$\mathbf y=\mathbf0$, which is not a projective point.
Thus the projective conic is nonsingular. Define the primitive
local factors and their product by
\begin{equation}\label{eq:sigma}
\sigma_p(\bx)=
\lim_{k\to\infty}p^{-2k}
\#\{\by\bmod p^k:p\nmid\by,\ q_{\bx}(\by)\equiv0\pmod{p^k}\},
\qquad
\SSS(\bx)=\prod_p\sigma_p(\bx).
\end{equation}
Proposition~\ref{lem:bad-local} ensures that the limits in
\eqref{eq:sigma} exist.
If $p\nmid 2x_0x_1x_2$, then $q_{\bx}\bmod p$ is nondegenerate
and has $p^2-1$ nonzero zeros by \cite[Theorem~6.27]{LN}.
These zeros are nonsingular, so \cite[Lemma~3.2]{Hanke} gives
\[
\sigma_p(\bx)
=\lim_{k\to\infty}p^{-2k}(p^2-1)p^{2(k-1)}
=1-p^{-2}.
\]
Thus the product defining
$\SSS(\bx)$ converges; it is zero if one of its finitely many
exceptional factors is zero.
Define
\begin{equation}\label{eq:real}
s(\bx)=\lim_{\epsilon\downarrow0}\frac1{2\epsilon}
\int_{[-1,1]^3}\one_{\{|q_{\bx}(\by)|\le\epsilon\}}\,d\by.
\end{equation}
We use the same definition for real coefficient vectors with
$x_0x_1x_2\ne0$. In integrals over real coefficient space, set $s(\bx)=0$
on the measure-zero coordinate planes.
For real coefficient vectors with $x_0x_1x_2\ne0$, we have
$s(\bx)=0$ for definite forms and $s(\bx)>0$ for indefinite forms.
For $t>0$, we have
$s(t\bx)=t^{-1}s(\bx)$.

For a nonsingular integral ternary quadratic form $q$, the notation
$\sigma_p(q)$ and $\SSS(q)$ means the same primitive local densities
as in \eqref{eq:sigma}, with $q_{\bx}$ replaced by $q$.
The same congruence limits define $\sigma_p(\bx)$ for $p$-adic
coefficient vectors with $x_0x_1x_2\ne0$, as needed in the coefficient
integrals below. Values on the measure-zero coordinate planes will
be set to zero in those integrals.
For a fixed smooth compactly supported real-valued function $w$, define
\[
s(q;w)=\lim_{\epsilon\downarrow0}\frac1{2\epsilon}
\int_{\R^3}w(\by)\one_{\{|q(\by)|\le\epsilon\}}\,d\by.
\]
Thus $s(\bx)=s(q_{\bx};\one_{[-1,1]^3})$, with the sharp weight interpreted through the same limit. If \(q\) is in addition primitive, write
\begin{equation}\label{PV}
P_{q,w}(Y)
=
\sum_{\substack{
\mathbf y\in\mathbb Z_{\mathrm{prim}}^3\\
q(\mathbf y)=0
}}
w(\mathbf y/Y),
\qquad
\nu_{q,w}
=
\frac12 s(q;w)\mathfrak S(q).
\end{equation}

\begin{lemma}\label{prop:fixed}
For each fixed nonsingular indefinite primitive integral quadratic form
$q_{\bx}$,
\begin{equation}\label{eq:fixed}
\#\{\by\in\Z_{\prim}^3:q_{\bx}(\by)=0,\ |\by|\le Y\}
=\frac12s(\bx)\SSS(\bx)Y+o_{\bx}(Y).
\end{equation}
Consequently, if
\begin{equation*}
 C_{[\mathbf x]}(Y)=\#\{[\mathbf y]\in\mathbb P^2(\mathbb Q):
 q_{\mathbf x}(\mathbf y)=0,\ y_0y_1y_2\ne0,\ |\mathbf y|\le Y\},
\end{equation*}
then
\begin{equation}\label{eq:kappa}
 C_{[\mathbf x]}(Y)=\kappa_{[\mathbf x]}Y+o_{\mathbf x}(Y),
 \qquad \kappa_{[\mathbf x]}=\tfrac14s(\mathbf x)\SSS(\bx).
\end{equation}
If the conic has no rational point, both the count and the leading
coefficient vanish.
\end{lemma}
\begin{proof}
This is the result in \cite[Introduction]{HB} and
\cite[Main Theorem 1.1]{Sofos}.
The factor $1/2$ in \eqref{eq:fixed} is already present in the count of
primitive integral vectors. On noting that $\by$ and $-\by$ represent the same point, one has
\[
 C_{[\mathbf x]}(Y)
 =\frac12\left(\frac12s(\mathbf x)\SSS(\bx)Y
                  +o_{\mathbf x}(Y)\right)+O(1)
 =\frac14s(\mathbf x)\SSS(\bx)Y+o_{\mathbf x}(Y)=\kappa_{[\mathbf x]}Y+o_{\mathbf x}(Y),
\]
it follows that \eqref{eq:kappa} is true.
The Hasse--Minkowski theorem \cite[Theorem 8]{Serre} ensures that a conic
without a rational point is insoluble at some place.
Removing $y_i=0$ removes only the finitely many projective points obtained
by intersecting the conic with three lines.
\end{proof}
We use \eqref{eq:kappa} for every primitive coefficient vector with
$x_0x_1x_2\ne0$, including the definite case, when $\kappa_{[\bx]}=0$.
Since $s(-\bx)=s(\bx)$ and $\sigma_p(-\bx)=\sigma_p(\bx)$, this
definition is independent of the choice of primitive representative.

\begin{lemma}\label{prop:linear}
Let $\mathbf w\in\Z_{\prim}^3$ and $X_i\ge1$. Then
\[
\#\{\bx\in\Z_{\prim}^3:|x_i|\le X_i,\ \mathbf w\cdot\bx=0\}
\le
12\pi\frac{X_0X_1X_2}{\max_i(|w_i|X_i)}+4.
\]
\end{lemma}
\begin{proof}
See Heath-Brown's bound for primitive solutions of a linear equation,
as quoted in \cite[Lemma 1]{LB}.
\end{proof}

\begin{lemma}\label{prop:conic-input}
Let $q$ be a primitive nonsingular integral ternary quadratic form.
Write $q(\by)=\tfrac12\by^{\mathsf T}Q_q\by$, where $Q_q$ is
symmetric and integral, with even diagonal entries. Use the conventions
\[
D_q=\tfrac12\det Q_q,\qquad
\|q\|=\left(\sum_{i,j}(Q_q)_{ij}^2\right)^{1/2},
\]
and let $\delta_q$ be the positive highest common factor of the \(2 \times 2\) minors of $Q_q$.
Put
$N(q,Y)=\#\{\by\in\Z_{\prim}^3:q(\by)=0,\ |\by|\le Y\}$.
Then, for $Y\ge1$,
\begin{equation}\label{eq:conic-upper}
N(q,Y)\ll
\tau(|D_q|)\left(1+\frac{Y\sqrt{\delta_q}}{|D_q|^{1/3}}\right).
\end{equation}
If the equation \(q(\mathbf z)=0\) has a nonzero rational solution,
let \(z_q\) be a nonzero integer solution with the smallest Euclidean
length. Then $\|z_q\|_2 \ll \|q\|.$

For the following uniform asymptotic, assume in addition that $q$ is
soluble over $\Q$. If the fixed smooth weight $w$ is supported in the
unit Euclidean ball, $Y\ge2$, and $s(q;w)>0$, Heath-Brown's uniform
formula is
\begin{equation}\label{eq:HB-input}
P_{q,w}(Y)=
\nu_{q,w}Y
+O_w\!\left(
\nu_{q,w}Y\,4^{\omega(|D_q|)}\rho
\frac{s(q;w_0)}{s(q;w)}
\left(\frac{\|q\|^2}{\|z_q\|_2Y}\right)^{1/4}\log(2Y)
+\tau(|D_q|)
\right),
\end{equation}
where $\rho=\|q\|^3/|D_q|$, and
$w_0(\by)=\exp(-1/(1-\|\by\|_2^2))$ inside the unit ball and is zero
outside it. These are the matrix conventions of \cite{HB}. In particular,
for $q_{\bx}(\by)=\sum_{i=0}^2 x_i y_i^2$, we have
\[
D_{q_{\bx}}=4x_0x_1x_2,\qquad
\|q_{\bx}\|=2\|\bx\|_2,\qquad
\delta_{q_{\bx}}=4\gcd(x_0x_1,x_0x_2,x_1x_2).
\]
\end{lemma}
\begin{proof}
The estimate \eqref{eq:conic-upper} is the uniform upper bound in
\cite[Corollary 2]{BHB}, also stated in full in
\cite[Introduction]{Sofos}.
The short-zero bound and \eqref{eq:HB-input} are given in
\cite[Theorems 2 and 6]{HB}, respectively.

\end{proof}
\begin{xremark}
If \(D_q<0\), apply Heath-Brown's results to \(-q\), whose
discriminant is \(-D_q>0\). The solutions, local densities,
and all quantities in the stated bounds are unchanged.
Section~\ref{sec:avg-error} below explains how to treat $s(q;w)=0$
or $s(q;w)<0$, and how to remove the support restriction on $w$.
\end{xremark}

\section{Peyre's constant}\label{sec:constant}

\begin{proposition}\label{prop:peyre}
The variety $V$ is a smooth split Fano threefold, whose Picard rank is $2$, and
\[
\Pic(V)=\Z H_x\oplus\Z H_y,\qquad
-K_V=2H_x+H_y,\qquad
\alpha(V)=\frac12,\quad\beta(V)=1.
\]
For the height $H$ defined above, its Peyre constant is
$\cP$, as defined in Theorem~\ref{thm:main}.
\end{proposition}
\begin{proof}
Consider the map of vector bundles
\[
 \mathcal O_{\mathbb P^2_{\mathbb Z}}^{\oplus3}
 \longrightarrow\mathcal O_{\mathbb P^2_{\mathbb Z}}(2),
 \qquad(a_i)\longmapsto\sum_i a_i y_i^2.
\]
At every point some $y_i$ is invertible on the corresponding
projective chart. On that chart, trivialize $\mathcal O(2)$ by
$y_i^2$; the coefficient of $a_i$ then becomes one. Thus one can
solve for $a_i$ with the other two coordinates free. The map is
therefore surjective, so its kernel
is locally free of rank two. Its projectivization, with the convention
that projective bundles parametrize lines, is exactly the integral
model defined by the equation of $V$. It follows that this model
is smooth over $\mathbb Z$, including at $2$, and that
$V\to\mathbb P^2_y$ is a $\mathbb P^1$-bundle.

The projective-bundle formula for the Picard group gives one
generator from the base and one from the relative hyperplane class.
They are $H_y$ and $H_x$, respectively, because the kernel is a
subbundle of the trivial rank-three bundle. Adjunction, using the
bidegree $(1,2)$ of the equation, gives
\[
 \operatorname{Pic}(V)=\mathbb ZH_x\oplus\mathbb ZH_y,
 \qquad
 K_V=(-3H_x-3H_y+H_x+2H_y)|_V=-2H_x-H_y;
\]
see \cite[Chapter II, \S\S7--8]{Hartshorne1977}.
The restriction of $\mathcal O(2,1)$ is ample, so $V$ is Fano.
Applying the projective-bundle formula over both $\mathbb Q$
and $\overline{\mathbb Q}$ gives
\[
 \operatorname{Pic}(V)=\mathbb ZH_x\oplus\mathbb ZH_y,
 \qquad
 \operatorname{Pic}(V_{\overline{\mathbb Q}})
 =\mathbb ZH_x\oplus\mathbb ZH_y,
\]
where the generators in the second group denote the base changes
of those in the first. Indeed, $H_x$ and $H_y$ are represented by
$\mathcal O_V(1,0)$ and $\mathcal O_V(0,1)$, respectively, both
defined over $\mathbb Q$. Thus the natural base-change map
\[
 \operatorname{Pic}(V)
 \longrightarrow
 \operatorname{Pic}(V_{\overline{\mathbb Q}})
\]
sends a basis to a basis and is an isomorphism. This proves that
$V$ is split over $\mathbb Q$.

To determine the effective cone, write $\mathcal O(a,b)$ for the
line bundle on $\mathbb P^2\times\mathbb P^2$. For any integers $a,b$,
the defining equation gives the restriction sequence
\[
 0\longrightarrow\mathcal O(a-1,b-2)
 \longrightarrow\mathcal O(a,b)
 \longrightarrow\mathcal O_V(a,b)\longrightarrow0.
\]
The cohomology of projective space gives
$H^1(\mathbb P^2,\mathcal O(r))=0$ for every integer $r$.
Specifically, the degree-one part of K\"unneth's formula is the
direct sum
\[
 \begin{split}
 &H^1(\mathbb P^2,\mathcal O(r))\otimes
       H^0(\mathbb P^2,\mathcal O(s))\\
 &\qquad\oplus
 H^0(\mathbb P^2,\mathcal O(r))\otimes
       H^1(\mathbb P^2,\mathcal O(s)),
 \end{split}
\]
and both summands are zero. Hence
\[
 H^1(\mathbb P^2\times\mathbb P^2,\mathcal O(r,s))=0
 \qquad(r,s\in\mathbb Z);
\]
see \cite[Chapter III, \S5]{Hartshorne1977}.
The long exact sequence therefore makes
$H^0(\mathcal O(a,b))\to H^0(\mathcal O_V(a,b))$ surjective.
If $a<0$ or $b<0$, the former space is zero, since one of the
two factors has no global sections. Hence every effective class
$aH_x+bH_y$ has $a,b\ge0$. Conversely, each generator is effective.
Thus
\[
 \overline{\operatorname{Eff}}(V)
   =\mathbb R_{\ge0}H_x+\mathbb R_{\ge0}H_y.
\]
Give the dual Picard lattice covolume one. Peyre's normalization
\cite{Peyre}, in rank two, is twice the volume of the truncated
dual cone. A dual vector with values $u,v$ on $H_x,H_y$ is
nonnegative on the effective cone exactly when $u,v\ge0$.
Its value on $-K_V$ is $2u+v$. Thus the truncation is the triangle
$0\le u\le1/2$, $0\le v\le1-2u$, and the lattice normalization
gives the ordinary measure $du\,dv$. Therefore
\[
 \alpha(V)=2\int_0^{1/2}\int_0^{1-2u}dv\,du
 =\frac12.
\]
The Galois action on the Picard group is trivial. Its first Galois
cohomology is zero: a continuous homomorphism from a profinite
group to the discrete group $\mathbb Z^2$ has finite image, which
must be zero. This proves $\beta(V)=1$.

The rank-two bundle is trivial over a nonempty open subset of
its base. Taking an affine chart in its projective-line fiber
exhibits a dense open subset of $V$ isomorphic to an open subset
of $\mathbb A^3$: intersect the trivializing subset of the base
with an affine chart of $\mathbb P^2$, then take its product
with $\mathbb A^1\subset\mathbb P^1$.
The projective-bundle description also shows that $V$ is
geometrically integral. It is smooth and projective by the same
description, so it satisfies the smoothness, properness, and
geometric integrality hypotheses of the rational-variety result
used below. Hence $V$ is rational, and
$\operatorname{Br}(V)=\operatorname{Br}(\mathbb Q)$ by
\cite[Proposition~6.9.10]{Poonen2017}.
Weak approximation on $\mathbb A^3$ also gives weak approximation
on $V$: the displayed open subset is locally dense in the smooth
variety $V$ at every place, so each finite set of local targets
can first be moved into that open subset and then approximated
there. Since $V$ is proper, its adelic topology is the product
topology for this purpose. Thus the closure of $V(\mathbb Q)$ is
the full adelic space, and no Brauer--Manin restriction changes
the mass entering Peyre's constant.

The cohomological conditions for the split Fano constant can
also be checked from the same restriction sequence. Set $a=b=0$.
All cohomology groups of $\mathcal O(-1)$ on $\mathbb P^2$ vanish,
and K\"unneth's formula therefore gives
$H^j(\mathbb P^2\times\mathbb P^2,\mathcal O(-1,-2))=0$ for
every $j$. The restriction sequence identifies
$H^i(V,\mathcal O_V)$ with
$H^i(\mathbb P^2\times\mathbb P^2,\mathcal O)$ for $i=1,2$;
both groups are zero by the projective-space cohomology already
cited. Together with the free split Picard group, the effective
cone calculation, and the full adelic closure, this justifies
using Peyre's product of the factors computed here.

For each prime $p$, every fiber over a point of
$\mathbb P^2(\mathbb F_p)$ is a projective line. There are
$(p^3-1)/(p-1)=p^2+p+1$ points on the base and
$(p^2-1)/(p-1)=p+1$ points on each fiber. Consequently
\[
 \#V(\mathbb F_p)=(p^2+p+1)(p+1).
\]
The standard sup-norm metrics are the model metrics on
$\mathcal O(2,1)\simeq-K_V$. Since the model is smooth of relative
dimension three, each residue class above an $\mathbb F_p$-point
has model measure $p^{-3}$. Properness ensures that every
$\mathbb Q_p$-point extends to a $\mathbb Z_p$-point. Thus the
local mass is $\#V(\mathbb F_p)/p^3$ before regularization, as in
the good-reduction formula for Peyre's measure \cite{Peyre}.
The split Picard
group has rank two, so Peyre's regularizing factor is
$(1-p^{-1})^2$. Therefore
\begin{align}
 \Delta
 &=\prod_p(1-p^{-1})^2\frac{\#V(\mathbb F_p)}{p^3}\notag
 =\prod_p(1-p^{-1})^2(1+p^{-1}+p^{-2})(1+p^{-1})\notag\\
 &=\prod_p(1-p^{-3})(1-p^{-2})
   =\frac1{\zeta(2)\zeta(3)}.
 \label{eq:finite-Peyre-product}
\end{align}
There is no additional residue factor: the Picard $L$-function
is $\zeta(s)^2$, and $\lim_{s\to1}(s-1)^2\zeta(s)^2=1$.

For $\bt\ne0$, define
\begin{equation}\label{J}
J(\bt)=
\frac{\operatorname{area}\{\mathbf u\in[-1,1]^3:
\mathbf u\cdot(t_0^2,t_1^2,t_2^2)=0\}}
{\sqrt{t_0^4+t_1^4+t_2^4}},\quad I:=\int_{[-1,1]^3}J(\bt)\,d\bt.
\end{equation}
Let $\tau_\infty$ be the total real measure induced by the
anticanonical metric defining $H$.
On the chart $x_0y_0y_2\ne0$, write
\[
[\bx]=[1:u:v],\qquad [\by]=[1:t:w],\qquad
v=-\frac{1+ut^2}{w^2},
\]
and put
\[
M_x=\max\{1,|u|,|v|\},\qquad
M_y=\max\{1,|t|,|w|\}.
\]
The local residue form is $du\,dt\,dw/w^2$.
The sup-norm anticanonical metric therefore gives
\[
d\tau_\infty=\frac{du\,dt\,dw}{|w|^2M_x^2M_y}.
\]
On the affine double cone, the substitutions
$\bx=r(1,u,v)$ and $\by=s(1,t,w)$, with $r,s>0$,
give the Leray measure
\[
\frac{r\,dr\,ds\,du\,dt\,dw}{|w|^2}.
\]
The four choices of signs of $x_0$ and $y_0$ have equal
contributions. Since the omitted coordinate divisors have measure
zero, integration over the sup-norm unit boxes yields
\[
\begin{aligned}
I
&=4\int_{\R^2\times\R^\times}
\left(\int_0^{1/M_x}r\,dr\right)
\left(\int_0^{1/M_y}ds\right)
\frac{du\,dt\,dw}{|w|^2}\\
&=2\int_{\R^2\times\R^\times}
\frac{du\,dt\,dw}{|w|^2M_x^2M_y}
=2\tau_\infty.
\end{aligned}
\]
Consequently,
\[
\alpha(V)\beta(V)\tau_\infty\Delta=I\Delta/4.
\]
Proposition~\ref{lem:I} gives $I=48\pi-56$, so
$I\Delta/4=\cP$, completing the proof.
\end{proof}
\begin{xremark}
Following the effective-cone viewpoint in
\cite[\S~2.2]{HL2025}, we explain the proportions in
Theorem~\ref{thm:split}.
Proposition~\ref{prop:peyre} gives
\[
\overline{\operatorname{Eff}}(V)
=\R_{\ge0}H_x+\R_{\ge0}H_y,
\qquad
-K_V=2H_x+H_y.
\]
In the dual coordinates $(u,v)$, the cone is
$\R_{\ge0}^2$, and its exponential measure is
$e^{-2u-v}\,du\,dv$. Its total mass is
\[
\int_0^\infty\int_0^\infty e^{-2u-v}\,dv\,du
=\frac12=\alpha(V).
\]
The two subcones $v\le u$ and $v>u$ therefore have
relative masses
\[
\begin{aligned}
\frac{1}{\alpha(V)}
\int_0^\infty\int_0^u e^{-2u-v}\,dv\,du
&=2\left(\frac12-\frac13\right)=\frac13,\\
\frac{1}{\alpha(V)}
\int_0^\infty\int_u^\infty e^{-2u-v}\,dv\,du
&=\frac23.
\end{aligned}
\]

Under the logarithmic height map
\[
([\bx],[\by])\longmapsto
\bigl(\log|\bx|,\log|\by|\bigr),
\]
the condition $v\le u$ is precisely $|\by|\le|\bx|$.
We now compare it with the cutoff in
Theorem~\ref{thm:split}.
If $H([\bx],[\by])\le B$ and $|\by|\le|\bx|$, then
\[
|\by|^3\le|\bx|^2|\by|\le B.
\]
Conversely,
\[
|\bx|<|\by|\le B^{1/3}
\quad\Longrightarrow\quad
H([\bx],[\by])=|\bx|^2|\by|\le B.
\]
Consequently,
\[
\begin{aligned}
&N_{\mathrm{low}}(B)
-\#\bigl\{([\bx],[\by])\in U(\Q):
 H([\bx],[\by])\le B,\ |\by|\le|\bx|\bigr\}\\
&\qquad=
\#\bigl\{([\bx],[\by])\in U(\Q):
 |\bx|<|\by|\le B^{1/3}\bigr\}.
\end{aligned}
\]

For a primitive integral vector $\by$, the vector
$(y_0^2,y_1^2,y_2^2)$ is primitive.
Applying Lemma~\ref{prop:linear} with
\[
\mathbf w=(y_0^2,y_1^2,y_2^2),
\qquad X_0=X_1=X_2=|\by|,
\]
gives
\[
\#\bigl\{\bx\in\Z_{\prim}^3:
 |\bx|\le|\by|,\ F(\bx,\by)=0\bigr\}
\le
12\pi\frac{|\by|^3}{|\by|^3}+4
=12\pi+4.
\]
Hence
\[
\begin{aligned}
\#\bigl\{([\bx],[\by])\in U(\Q):
 |\bx|<|\by|\le B^{1/3}\bigr\}
\ll
\sum_{\substack{\by\in\Z_{\prim}^3\\
                 |\by|\le B^{1/3}}}1
\ll B.
\end{aligned}
\]
Thus replacing the cutoff $|\by|\le B^{1/3}$ by
$|\by|\le|\bx|$ changes the count by only $O(B)$,
which is negligible compared with $B\log B$. The cone masses are therefore consistent with the
proportions in Theorem~\ref{thm:split}.
\end{xremark}

\begin{proposition}\label{lem:I}
With $s(\bt)$ and $J(\bt)$ as in \eqref{eq:real} and \eqref{J}
respectively, we have
\[
\int_{[-1,1]^3}s(\bt)\,d\bt
=\int_{[-1,1]^3}J(\bt)\,d\bt
=:I=48\pi-56.
\]
\end{proposition}
\begin{proof}
Fix \(\mathbf t\ne\mathbf0\). Choose
\(\{i,j,k\}=\{0,1,2\}\) such that
\[
A=t_i^2=\max_{0\le\ell\le2}t_\ell^2>0,
\qquad
b=t_j^2,\qquad c=t_k^2.
\]
Thus \(0\le b,c\le A\). The plane defining \(J(\mathbf t)\)
has equation
\[
Ax_i+bx_j+cx_k=0,
\qquad
x_i=-\frac{bx_j+cx_k}{A}.
\]
Its intersection with \([-1,1]^3\) is therefore parametrized by
\[
\mathcal D
=
\left\{
(x_j,x_k)\in[-1,1]^2:
|bx_j+cx_k|\le A
\right\}.
\]
The Euclidean surface area element is
\[
dS
=
\sqrt{1+\frac{b^2}{A^2}+\frac{c^2}{A^2}}
\,dx_j\,dx_k
=
\frac{\sqrt{A^2+b^2+c^2}}{A}\,dx_j\,dx_k.
\]
Since
\(\sqrt{t_0^4+t_1^4+t_2^4}=\sqrt{A^2+b^2+c^2}\),
the definition of \(J\) gives
\[
J(\mathbf t)
=
\frac{\operatorname{area}(\mathcal D)}{A}
=
\frac1A
\int_{[-1,1]^2}
\mathbf 1_{\{|bx_j+cx_k|\le A\}}\,dx_j\,dx_k.
\]

If \(b+c\le A\), then
\[
|bx_j+cx_k|\le b+c\le A
\qquad ((x_j,x_k)\in[-1,1]^2),
\]
so \(\operatorname{area}(\mathcal D)=4\).
If \(b+c>A\), then \(bc>0\), and the excluded regions
\[
bx_j+cx_k>A,
\qquad
bx_j+cx_k<-A
\]
are two congruent right triangles at the opposite corners
\((1,1)\) and \((-1,-1)\).
The first triangle has vertices
\[
(1,1),\qquad
\left(1,\frac{A-b}{c}\right),\qquad
\left(\frac{A-c}{b},1\right).
\]
Its two perpendicular sides have lengths
\((b+c-A)/c\) and \((b+c-A)/b\).
Hence the total excluded area is
\[
2\cdot\frac12
\frac{b+c-A}{c}\frac{b+c-A}{b}
=
\frac{(b+c-A)^2}{bc}.
\]
Combining the two cases yields
\begin{equation}\label{J1}
J(\mathbf t)
=
\frac1A
\left(
4-\frac{(b+c-A)_+^2}{bc}
\right).
\end{equation}
Here \(z_+=\max(z,0)\), and the quotient is defined to be zero
when \(bc=0\). Indeed, in that case \(b+c\le A\), so
\(J(\mathbf t)=4/A\).
This convention also agrees with continuous extension:
for \(0<b,c\le A\),
\[
0\le
\frac{(b+c-A)_+^2}{bc}
\le
\frac{2\min(b,c)}{A}.
\]
To see this, when \(b+c>A\) we have
\(b+c-A\le\min(b,c)\) and \(\max(b,c)>A/2\);
otherwise the quotient is zero.

Since \(\mathcal D\subseteq[-1,1]^2\), we have
\[
0\le J(\mathbf t)
\le\frac4A
=\frac4{|\mathbf t|^2},
\qquad
|\mathbf t|=\max_{0\le i\le2}|t_i|.
\]
Thus \(J(\mathbf t)\ll|\mathbf t|^{-2}\), with an absolute
implied constant. Moreover,
$\operatorname{vol}
\{\mathbf t\in\mathbb R^3:|\mathbf t|\le r\}=8r^3.$
Consequently, for \(0<\rho\le1\),
\[
\begin{aligned}
\int_{0<|\mathbf t|\le\rho}J(\mathbf t)\,d\mathbf t
\le
4\int_{0<|\mathbf t|\le\rho}
|\mathbf t|^{-2}\,d\mathbf t
=
4\int_0^\rho r^{-2}\cdot24r^2\,dr
=96\rho.
\end{aligned}
\]
In particular, \(J\in L^1([-1,1]^3)\).
The value assigned to \(J(\mathbf0)\) does not affect the integral.

The definition of \(J\) is invariant under coordinate sign
changes and permutations. The sets where two coordinates
have equal absolute value have measure zero. Therefore,
\[
\begin{aligned}
I=\int_{[-1,1]^3}J(\mathbf t)\,d\mathbf t
=8\int_{[0,1]^3}J(\mathbf t)\,d\mathbf t
=24\int_0^1\int_0^{t_0}\int_0^{t_0}
J(t_0,t_1,t_2)\,dt_2\,dt_1\,dt_0.
\end{aligned}
\]
The factor \(24=8\cdot3\) accounts for the eight sign patterns
and the three possible positions of the largest coordinate.

On the region \(0\le t_1,t_2\le t_0\le1\), make the substitution
\[
t_0=r,\qquad t_1=ru,\qquad t_2=rv,
\qquad
0<r\le1,\quad 0\le u,v\le1.
\]
Its Jacobian is
\[
\left|
\det\frac{\partial(t_0,t_1,t_2)}{\partial(r,u,v)}
\right|
=
\left|
\det
\begin{pmatrix}
1&0&0\\
u&r&0\\
v&0&r
\end{pmatrix}
\right|
=r^2.
\]
Here \(A=r^2\), \(b=r^2u^2\), and \(c=r^2v^2\), so the
formula for \(J\) gives
\[
J(r,ru,rv)
=
\frac1{r^2}
\left(
4-\frac{(u^2+v^2-1)_+^2}{u^2v^2}
\right),
\]
with the quotient interpreted as zero when \(uv=0\).
Hence
\[
\begin{aligned}
I
&=
24\int_0^1\int_0^1\int_0^1
J(r,ru,rv)\,r^2\,du\,dv\,dr\\
&=
24\left(\int_0^1dr\right)
\int_0^1\int_0^1
\left(
4-\frac{(u^2+v^2-1)_+^2}{u^2v^2}
\right)\,du\,dv\\
&=
24\int_0^1\int_0^1
\left(
4-\frac{(u^2+v^2-1)_+^2}{u^2v^2}
\right)\,du\,dv.
\end{aligned}
\]

The integral of the subtracted term is
\[
\int_0^1\int_{\sqrt{1-u^2}}^1
\frac{(u^2+v^2-1)^2}{u^2v^2}\,dv\,du
=\int_0^1\left(4-u^2+
\frac8{3u^2}\bigl((1-u^2)^{3/2}-1\bigr)\right)du.
\]
Integration by parts gives
\[
\int_0^1\frac{(1-u^2)^{3/2}-1}{u^2}\,du
=1-\frac{3\pi}{4}.
\]
The subtracted term is therefore $19/3-2\pi$, whence $I=48\pi-56$.

Finally, we will prove $$\int_{[-1,1]^3}s(\bt)\,d\bt =\int_{[-1,1]^3}J(\bt)\,d\bt.$$
Fix \(\varepsilon>0\) and
\(\mathbf x\in[-1,1]^3\) with \(x_0x_1x_2\ne0\).
Choose \(\{i,j,k\}=\{0,1,2\}\) such that \(x_i x_j>0\).
For fixed \(y_k\), rescaling \(y_i,y_j\) by
\(\sqrt{|x_i|},\sqrt{|x_j|}\), respectively, places the
admissible region inside a disk or annulus of area at most
\(2\pi\varepsilon\). Consequently,
\[
\begin{aligned}
\frac{1}{2\varepsilon}
\int_{[-1,1]^3}
\mathbf 1_{\{|F(\mathbf x,\mathbf y)|\le\varepsilon\}}
\,d\mathbf y
\le
\frac{1}{2\varepsilon}
\int_{-1}^{1}
\frac{2\pi\varepsilon}{\sqrt{|x_i x_j|}}\,dy_k
=
\frac{2\pi}{\sqrt{|x_i x_j|}}
\le
\frac{2\pi}{\sqrt{|x_0x_1x_2|}}.
\end{aligned}
\]
By the definition of \(s\), the left-hand side converges to
\(s(\mathbf x)\) almost everywhere as \(\varepsilon\downarrow0\).

Next, fix \(\mathbf y\ne\mathbf0\), and choose
\(\{i,j,k\}=\{0,1,2\}\) with
\(y_i^2=\max_\ell y_\ell^2\).
For fixed \(x_j,x_k\), the condition
\(|F(\mathbf x,\mathbf y)|\le\varepsilon\) restricts \(x_i\)
to an interval of length at most \(2\varepsilon/y_i^2\).
Thus
\[
0\le
\frac{1}{2\varepsilon}
\int_{[-1,1]^3}
\mathbf 1_{\{|F(\mathbf x,\mathbf y)|\le\varepsilon\}}
\,d\mathbf x
\le \frac{4}{y_i^2}
=\frac{4}{|\mathbf y|^2}.
\]
Moreover, one-dimensional integration followed by dominated
convergence gives
\[
\begin{aligned}
\lim_{\varepsilon\downarrow0}
\frac{1}{2\varepsilon}
\int_{[-1,1]^3}
\mathbf 1_{\{|F(\mathbf x,\mathbf y)|\le\varepsilon\}}
\,d\mathbf x
=
\frac{1}{y_i^2}
\int_{[-1,1]^2}
\mathbf 1_{\{|x_jy_j^2+x_ky_k^2|\le y_i^2\}}
\,dx_j\,dx_k
=J(\mathbf y).
\end{aligned}
\]
Indeed, the boundary
\(|x_jy_j^2+x_ky_k^2|=y_i^2\) has planar measure zero.
The last equality follows by writing the plane as
$
x_i=-\frac{x_jy_j^2+x_ky_k^2}{y_i^2},$
whose Euclidean area factor is
\(\sqrt{y_0^4+y_1^4+y_2^4}/y_i^2\).

Both majorants are integrable, since
\[
\int_{[-1,1]^3}
\frac{d\mathbf x}{\sqrt{|x_0x_1x_2|}}=64,
\qquad
\int_{[-1,1]^3}
\frac{d\mathbf y}{|\mathbf y|^2}=24.
\]
Therefore, Tonelli's theorem and dominated convergence yield
\[
\begin{aligned}
\int_{[-1,1]^3}s(\mathbf t)\,d\mathbf t
&=
\lim_{\varepsilon\downarrow0}
\frac{1}{2\varepsilon}
\int_{[-1,1]^3}\int_{[-1,1]^3}
\mathbf 1_{\{|F(\mathbf x,\mathbf y)|\le\varepsilon\}}
\,d\mathbf y\,d\mathbf x\\
&=
\lim_{\varepsilon\downarrow0}
\frac{1}{2\varepsilon}
\int_{[-1,1]^3}\int_{[-1,1]^3}
\mathbf 1_{\{|F(\mathbf x,\mathbf y)|\le\varepsilon\}}
\,d\mathbf x\,d\mathbf y\\
&=
\int_{[-1,1]^3}J(\mathbf t)\,d\mathbf t.
\end{aligned}
\]
\end{proof}

\section{The asymptotic in the low-\texorpdfstring{$y$}{y} region}\label{sec:low}

For a nonzero primitive vector $\by$, define the saturated rank-two lattice
\[
\Lambda_{\by}=\{\bx\in\Z^3:\bx\cdot(y_0^2,y_1^2,y_2^2)=0\},
\qquad
\lambda(\by)=\min_{0\ne\bx\in\Lambda_{\by}}\|\bx\|_2.
\]
Set $\tilde{\by}=(y_0^2,y_1^2,y_2^2)$, since $\gcd(y_0^2,y_1^2,y_2^2)=1$, we have
\begin{equation}\label{eq:lattice-det}
\det\Lambda_{\by}=\frac{\|\tilde{\by}\|_2}{\gcd(y_0^2,y_1^2,y_2^2)}=\sqrt{y_0^4+y_1^4+y_2^4}.
\end{equation}

The determinant $\det\Lambda_{\mathbf y}$ is the Euclidean area
of a fundamental parallelogram of $\Lambda_{\mathbf y}$ in its
two-dimensional span. Choosing orthonormal coordinates in this
plane preserves lengths and areas, so Minkowski's convex-body
theorem \cite[Chapter III, \S2]{Cassels1997} applies in dimension two. It states
that any convex body symmetric about the origin whose area
exceeds $4\det\Lambda_{\mathbf y}$ contains a nonzero point
of $\Lambda_{\mathbf y}$.

Fix $0<r<\lambda(\mathbf y)$ and consider the closed disk
$
 D_r=
 \left\{
 \mathbf x\in\operatorname{span}_{\mathbb R}\Lambda_{\mathbf y}:
 \|\mathbf x\|_2\le r
 \right\}.$
This disk is convex and symmetric about the origin, and
$\operatorname{area}(D_r)=\pi r^2.$
By the definition of $\lambda(\mathbf y)$, every
$\mathbf x\in\Lambda_{\mathbf y}\setminus\{\mathbf0\}$ satisfies
$\|\mathbf x\|_2\ge\lambda(\mathbf y)>r.$
Consequently,
$D_r\cap\Lambda_{\mathbf y}=\{\mathbf0\}.$
If $\pi r^2>4\det\Lambda_{\mathbf y}$, Minkowski's theorem would
give a nonzero lattice point in $D_r$, contradicting this
identity. Hence
\[
 \pi r^2\le4\det\Lambda_{\mathbf y}
 \qquad\bigl(0<r<\lambda(\mathbf y)\bigr).
\]
Letting $r$ increase to $\lambda(\mathbf y)$ gives
\[
 \pi\lambda(\mathbf y)^2
 =\lim_{r\uparrow\lambda(\mathbf y)}\pi r^2
 \le4\det\Lambda_{\mathbf y}.
\]
Hence, $\det\Lambda_{\mathbf y}\ge\frac\pi4\lambda(\mathbf y)^2$.

Furthermore, every nonzero integer vector has Euclidean norm
at least $1$, so $\lambda(\mathbf y)\ge1$. Using
\eqref{eq:lattice-det} and
$|\mathbf y|=\max_{0\le i\le2}|y_i|$, we obtain
\begin{align*}
\lambda(\mathbf y)^2
 \le\frac4\pi\det\Lambda_{\mathbf y}
 =\frac4\pi\sqrt{y_0^4+y_1^4+y_2^4}
 \le\frac{4\sqrt3}\pi|\mathbf y|^2.
\end{align*}
Thus
\begin{equation}\label{eq:lattice-minimum}
 \det\Lambda_{\mathbf y}\ge\frac\pi4\lambda(\mathbf y)^2,
 \qquad
 1\le\lambda(\mathbf y)\le 2\cdot3^{1/4}\pi^{-1/2}|\mathbf y|.
\end{equation}

\begin{proposition}\label{lem:lattice}
For $X\ge1$, let $J(\by)$ be defined in \eqref{J}, we have
\begin{equation}\label{eq:lattice}
P_{\by}(X):=
\#\{\bx\in\Lambda_{\by}\cap\Z_{\prim}^3:|\bx|\le X\}
=\frac{X^2J(\by)}{\zeta(2)}
+O\left(\frac{X\log(2X)}{\lambda(\by)}\right).
\end{equation}
\end{proposition}
\begin{proof}
Let
\[
H_{\mathbf y}
 =\{\mathbf x\in\mathbb R^3:
   \mathbf x\cdot(y_0^2,y_1^2,y_2^2)=0\},
\qquad
K_X=H_{\mathbf y}\cap[-X,X]^3.
\]
Since $K_X=XK_1$, the definition of $J(\mathbf y)$
and the formula for $\det\Lambda_{\mathbf y}$ give
\[
\frac{\operatorname{area}(K_X)}
     {\det\Lambda_{\mathbf y}}
 =X^2J(\mathbf y).
\]

Identify $H_{\mathbf y}$ isometrically with
$\mathbb R^2$. The boundary of $K_X$ consists of
at most six line segments, each of length at most
$2\sqrt{3}X$. Parametrizing each segment affinely
on $[0,1]$, we may apply
\cite[Theorem~2.4]{Widmer2012}
with $n=2$, $M\leq6$, and $L=2\sqrt{3}X$.
It follows that
\[
\left|
\#(\Lambda_{\mathbf y}\cap K_X)
-\frac{\operatorname{area}(K_X)}
       {\det\Lambda_{\mathbf y}}
\right|
\ll 1+\frac{X}{\lambda(\mathbf y)},
\]
with an absolute implied constant. Consequently,
\begin{equation}\label{eq:all-lattice-count}
\#(\Lambda_{\mathbf y}\cap K_X)
 =X^2J(\mathbf y)
  +O\!\left(\frac{X}{\lambda(\mathbf y)}+1\right).
\end{equation}

Saturation identifies coordinate primitivity with lattice
primitivity: if $\mathbf x=d\mathbf z$ with
$\mathbf x\in\Lambda_{\mathbf y}$, $\mathbf z\in\mathbb Z^3$
and $d\ge1$, then the defining linear equation gives
$\mathbf z\in\Lambda_{\mathbf y}$. Exclude the zero vector and put
$D=\lfloor\sqrt3X/\lambda(\mathbf y)\rfloor$.
Every vector in $K_{X/d}$ has Euclidean norm at most $\sqrt3X/d$.
For $d>D$ this is strictly smaller than $\lambda(\mathbf y)$,
so that section contains no nonzero lattice point. Then, we have
\begin{align*}
 P_{\mathbf y}(X)
 &=\sum_{\substack{\mathbf x\in\Lambda_{\mathbf y}\\
                         0<|\mathbf x|\le X}}
       \sum_{d\mid\gcd(x_0,x_1,x_2)}\mu(d)
 =\sum_{d\le D}\mu(d)
       \#\{\mathbf z\in\Lambda_{\mathbf y}:0<|\mathbf z|\le X/d\}\\
 &=\sum_{d\le D}\mu(d)
       \bigl(\#(\Lambda_{\mathbf y}\cap K_{X/d})-1\bigr).
\end{align*}
All sums before truncation are finite. In the second equality,
$\mathbf x=d\mathbf z$ is a bijection by saturation; the terms
with $d>D$ vanish by the preceding length bound.
Assume first that $D\ge1$. Substituting
\eqref{eq:all-lattice-count}, which also applies when $X/d<1$,
yields the following formula. Subtracting the zero vector adds
at most one to each discrepancy, already allowed by its constant
term:
\[
 P_{\mathbf y}(X)
 =X^2J(\mathbf y)\sum_{d\le D}\frac{\mu(d)}{d^2}
  +O\left(D+\frac X{\lambda(\mathbf y)}
                                      \sum_{d\le D}\frac1d\right).
\]
Since $\lambda(\mathbf y)\ge1$, one has
$D\le\sqrt3X/\lambda(\mathbf y)\le\sqrt3X$.
The harmonic sum is at most $1+\log D\ll\log(2X)$; thus this
error is $O(X\log(2X)/\lambda(\mathbf y))$, because
\[
 D+\frac X{\lambda(\mathbf y)}(1+\log D)
 \le\frac X{\lambda(\mathbf y)}
          \bigl(\sqrt3+1+\log(\sqrt3X)\bigr)
 \ll\frac{X\log(2X)}{\lambda(\mathbf y)}.
\]

The unit section $K_1$ lies in the disk of radius $\sqrt3$, so
$J(\mathbf y)\le3\pi/(\det\Lambda_{\mathbf y})$.
Moreover, $\lfloor u\rfloor\ge u/2$ for $u\ge1$, and hence
$D\ge\sqrt3X/(2\lambda(\mathbf y))$. The tail of the main series
is absolutely convergent, and
$\sum_{d>D}d^{-2}\le\int_D^\infty t^{-2}\,dt=D^{-1}$.
Its contribution is therefore bounded by
\[
 X^2J(\mathbf y)\sum_{d>D}\frac1{d^2}
 \ll\frac{X^2}{D\det\Lambda_{\mathbf y}}
 \ll\frac{X\lambda(\mathbf y)}{\det\Lambda_{\mathbf y}}
 \ll\frac X{\lambda(\mathbf y)},
\]
where the last step uses \eqref{eq:lattice-minimum}.
We may therefore extend the series and use
$\sum_{d\ge1}\mu(d)d^{-2}=\prod_p(1-p^{-2})=1/\zeta(2)$.
The Euler-product identity is applicable here because
$\sum_d|\mu(d)|d^{-2}\le\zeta(2)<\infty$.
If $D=0$, then $X/\lambda(\mathbf y)<1/\sqrt3$ and the count is
zero. The proposed main term is at most
\[
 O\left(\frac{X^2}{\det\Lambda_{\mathbf y}}\right)
 =O\left(\frac{X^2}{\lambda(\mathbf y)^2}\right)
 =O\left(\frac X{\lambda(\mathbf y)}\right).
\]
It is absorbed by the claimed error. This proves the proposition
for all $X\ge1$.
\end{proof}

\begin{lemma}\label{lem:short-lattice}
For $Y\ge2$, let $\lambda(\by)$ be defined as above, we have
\begin{equation}\label{eq:short-lattice}
\sum_{\substack{\by\in\Z_{\prim}^3,\ y_0y_1y_2\ne0\\
Y<|\by|\le2Y}}\frac1{\lambda(\by)}
\ll Y^2\log(2Y).
\end{equation}
\end{lemma}
\begin{proof}
From \eqref{eq:lattice-minimum}, we have $1\le\lambda(\by)\le CY$.
For $1\le L\le CY$, let $A_Y(L)$ count the above vectors satisfying $\lambda(\by)\le L$.
A shortest vector is necessarily primitive. If all three of its
coordinates are nonzero, the corresponding point belongs to $U$. Since $|\mathbf x|\le\|\mathbf x\|_2\le L$, one has
height at most $2L^2Y$. By the known upper bound \eqref{eq:LBupper},
this contribution is at most $ O\bigl(L^2Y\log(2L^2Y)\bigr)=O\bigl(L^2Y\log(2Y)\bigr)$.

If a shortest vector has one zero coordinate, then, up to finitely many
choices of signs and permutations, we may write
\[
\bx=(a^2,-b^2,0),\qquad
\by=(bh,ah,c),\qquad
a,b>0,\quad(a,b)=1,\quad h,c\in\mathbb Z\setminus\{0\}.
\]
For $t=\max(a,b)\ll\sqrt L$,
$(a,b)$ has $O(t)$ possible values. The bound $|\mathbf y|\le2Y$ gives
$0<|h|\le2Y/t$ and $0<|c|\le2Y$. For each pair $(a,b)$, the
number of $(h,c)$ is at most
\[
 2\lfloor2Y/t\rfloor\cdot2\lfloor2Y\rfloor
 \le\frac{16Y^2}{t}.
\]
Summation gives
\[
 \sum_{1\le t\le\sqrt L}O(t)\,O(Y^2/t)
 \ll Y^2\sqrt L.
\]
A nonzero $\bx$ with two zero coordinates cannot be orthogonal to a
vector $\by$ whose coordinates are all nonzero.
Thus
\[
A_Y(L)\ll L^2Y\log(2Y)+Y^2\sqrt L.
\]

Clearly, we have$\frac1{\lambda(\mathbf y)}
 =\frac1{CY}+\int_{\lambda(\mathbf y)}^{CY}t^{-2}\,dt.$
By partial summation, we have
\begin{align*}
 \sum\lambda(\mathbf y)^{-1}
 &=\frac{A_Y(CY)}{CY}
       +\int_1^{CY}\frac{A_Y(t)}{t^2}\,dt\\
 &\ll Y^2\log(2Y)+Y^{3/2}
       +Y\log(2Y)\int_1^{CY}dt
       +Y^2\int_1^{CY}t^{-3/2}\,dt\\
 &\ll Y^2\log(2Y).
\end{align*}
\end{proof}

\begin{proposition}\label{lem:Jmean}
For $Y\ge2$, let $J(\by)$ and $I$ be defined in \eqref{J}, we have
\[
\sum_{\substack{\by\in\Z_{\prim}^3,\ y_0y_1y_2\ne0\\
Y<|\by|\le2Y}}J(\by)=\frac{I}{\zeta(3)}Y+O(1).
\]
Consequently, for $T\geq 2$, we have
\begin{equation}\label{eq:Jharmonic}
\sum_{\substack{\by\in\Z_{\prim}^3,\ y_0y_1y_2\ne0\\
|\by|\le T}}\frac{J(\by)}{|\by|}
=\frac{I}{\zeta(3)}\log T+O(1).
\end{equation}
\end{proposition}
\begin{proof}
Equation \eqref{J1} shows that $J$ is Lipschitz on
$1\le|\bt|\le2$.
For the correction term $\frac{(b+c-A)_+^2}{bc}$, suppose without loss of generality that
$b\le c$. In its active region $b+c>A$, we have $c>A/2$ and
$0<b+c-A\le b$. These inequalities bound the function and all its
first derivatives in $A,b,c$, even as $b\downarrow0$.
The correction and its first derivatives vanish on the part of
$b+c=A$ where $bc>0$.
Composition with $t_i\mapsto t_i^2$, followed by the finite partition
according to the largest coordinate, therefore preserves the
Lipschitz property. Hence, for arbitrary $1\le|\mathbf u|, |\mathbf v|\le2$, we have
\[
|J(\mathbf u)-J(\mathbf v)|
\le L|\mathbf u-\mathbf v|.
\]
Together with the homogeneity
$J(R\bt)=R^{-2}J(\bt)$, this gives
\[
 \sum_{R<|\mathbf z|\le2R}J(\mathbf z)
 =\int_{R<|\mathbf t|\le2R}J(\mathbf t)\,d\mathbf t+O(1)
 =IR+O(1).
\quad(R\ge1/2).
\]
Removing vectors with a zero coordinate changes the sum by only $O(1)$.
Apply M\"obius inversion to the remaining vectors and write
$\mathbf y=d\mathbf z$. Every nonzero vector under consideration
has $d\le2Y$. Homogeneity and the preceding shell estimate, with
$R=Y/d\ge1/2$, give
\begin{align*}
 \sum_{\substack{\mathbf y\in\mathbb Z^3_{\rm prim}\\
             y_0y_1y_2\ne0,\ Y<|\mathbf y|\le2Y}}J(\mathbf y)
 &=\sum_{d\le2Y}\frac{\mu(d)}{d^2}
      \sum_{\substack{Y/d<|\mathbf z|\le2Y/d\\z_0z_1z_2\ne0}}
                            J(\mathbf z)\\
 &=IY\sum_{d\le2Y}\frac{\mu(d)}{d^3}
           +O\left(\sum_{d\le2Y}\frac1{d^2}\right)\\
 &=\frac{IY}{\zeta(3)}+O(1).
\end{align*}

Summing over dyadic intervals, we obtain
\[
\sum_{\substack{\by\in\Z_{\prim}^3,\ y_0y_1y_2\ne0\\
|\by|\le T}}J(\by)
=\frac{IT}{\zeta(3)}+O(\log(2T)).
\]
Partial summation gives
\begin{align*}
 \sum_{\substack{\mathbf y\in\mathbb Z^3_{\rm prim}\\
                  y_0y_1y_2\ne0,\ |\mathbf y|\le T}}
                   \frac{J(\mathbf y)}{|\mathbf y|}
 &=\frac1T\sum_{\substack{\mathbf y\in\mathbb Z^3_{\rm prim}\\
                  y_0y_1y_2\ne0,\ |\mathbf y|\le T}}J(\mathbf y)
   +\int_1^T\frac1{t^2}
      \sum_{\substack{\mathbf y\in\mathbb Z^3_{\rm prim}\\
                  y_0y_1y_2\ne0,\ |\mathbf y|\le t}}J(\mathbf y)\,dt\\
 &=\frac I{\zeta(3)}\log T
   +O\left(1+\frac{\log(2T)}T
                    +\int_1^T\frac{\log(2t)}{t^2}\,dt\right)\\
 &=\frac{I}{\zeta(3)}\log T+O(1).
\end{align*}
\end{proof}

\begin{proposition}\label{thm:low}
Equation \eqref{eq:low-main} holds.
\end{proposition}
\begin{proof}
Let $N_{\le T}(B)$ count the points of $U$ of height at most $B$
with $|\mathbf y|\le T$, where $2\le T\le B^{1/3}$.
For a fixed primitive $\mathbf y$ with nonzero coordinates,
the height inequality is
\[
 |\mathbf x|\le X=\sqrt{B/|\mathbf y|}.
\]
Here $X\ge B^{1/3}\ge1$, by Proposition~\ref{lem:lattice}, we have
\[
P_{\by}(\sqrt{B/|\mathbf y|})=
\frac{B}{\zeta(2)}
\frac{J(\by)}{|\mathbf y|}
+O\left(\frac{\sqrt{B/|\mathbf y|}\log (2\sqrt{B/|\mathbf y|})}{\lambda(\mathbf y)}\right).
\]

For $Y<|\mathbf y|\le2Y$, by Lemma \ref{lem:short-lattice}, the total lattice-point error is
\begin{align*}
 \sum_{\mathbf y}\frac{\sqrt{B/|\mathbf y|}\log (2\sqrt{B/|\mathbf y|})}{\lambda(\mathbf y)}
 \ll \sqrt B\,Y^{-1/2}\log(2B)
                \sum_{\mathbf y}\frac1{\lambda(\mathbf y)}
 \ll \sqrt B\,Y^{3/2}\log(2B)\log(2Y).
\end{align*}
Summing over dyadic intervals, their total error is
$O(\sqrt B\,T^{3/2}\log(2B)\log(2T))$.
For each fixed $\by$ with all coordinates nonzero,
each line $x_i=0$ contains only two opposite primitive generators.
Thus removing zero coordinates costs $O(T^3)$ in total.
The two projective sign quotients together divide the count by $4$,
so
\begin{align*}
N_{\le T}(B)&=
\frac{B}{4\zeta(2)}
\sum_{\substack{\by\in\Z_{\prim}^3,\ y_0y_1y_2\ne0\\|\mathbf y|\le T}}
\frac{J(\by)}{|\mathbf y|}
+O\left(\sqrt B\,T^{3/2}\log(2B)\log(2T)+T^3\right)\\
&=\frac{I B}{4\zeta(2)\zeta(3)}\log T
      +O\left(B+\sqrt B\,T^{3/2}\log(2B)\log(2T)+T^3\right)\\
 &=\cP B\log T+
      O\left(B+\sqrt B\,T^{3/2}\log(2B)\log(2T)+T^3\right),
\end{align*}
where the second equality uses Proposition \ref{lem:Jmean}.

Since $N_{\rm low}(B)=N_{\le B^{1/3}}(B)$, if take $T=B^{1/3}$ in $N_{\le T}(B)$, however, the error term is larger than the main term.
Now, we truncate the summation at $T_0=B^{1/3}/(\log B)^3$. Taking $T_0$ in $N_{\le T}(B)$, we obtain
\[
N_{\le T_0}(B)
=\frac{\cP}{3}B\log B-3\cP B\log\log B+O(B).
\]
It remains to bound $T_0<|\mathbf y|\le B^{1/3}$.
For $r=|\mathbf y|$, Lemma~\ref{prop:linear}, applied to the
primitive normal $(y_0^2,y_1^2,y_2^2)$ and equal side lengths
$\sqrt{B/r}$, gives
\[
 P_{\mathbf y}(\sqrt{B/r})
 \ll\frac{B/r}{r^2}+1=\frac B{r^3}+1.
\]
For integer height $r$, the number of vectors $\by$ is $O(r^2)$, and hence
\[
0\le N_{\le B^{1/3}}(B)-N_{\le T_0}(B)\ll\sum_{T_0<r\le B^{1/3}}r^2(B/r^3+1)
\ll B\log\log B.
\]
This proves \eqref{eq:low-main}.
\end{proof}

\section{Weighted squarefree averages and the two endpoints}\label{sec:sf}

Put
\begin{equation} \label{f}
 f(n)=\prod_{p\mid n}\frac p{p+1}.
\end{equation}
The empty product is one. All constants below are independent of
the chosen unit residue classes unless a dependence is stated.

\begin{lemma}\label{lem:small-char}
Let $\chi\pmod q$ be nonprincipal, with $(q,Q)=1$ and $(r,Q)=1$.
For $D_0,Z\geq1$ and $1/2<\sigma<1$, let $f(n)$ be defined in \eqref{f}, we have
 \[\sum_{\substack{n\leq Z\\n\equiv r\pmod Q\\(n,D_0)=1}}
 \mu^2(n)f(n)\chi(n)\ll_\sigma q\tau(D_0)Z^\sigma.
\]
The same estimate holds on any subinterval of $[1,Z]$.
\end{lemma}
\begin{proof}
Write $\mu^2(n)f(n)=\sum_{d\mid n}b(d)$. At prime powers the
left side is one at $p^0$, $p/(p+1)$ at $p$, and zero at $p^k$
for $k\geq2$. Subtracting consecutive partial sums gives
\[
 b(p)=-\frac1{p+1},\qquad b(p^2)=-\frac p{p+1},\qquad
 b(p^k)=0\quad(k\geq3),
\]
and hence
\begin{equation}
 \sum_{d\geq1}\frac{|b(d)|}{d^\sigma}
 =\prod_p\left(1+\frac{p^{-\sigma}}{p+1}
                    +\frac{p^{1-2\sigma}}{p+1}\right)<\infty.
 \label{arith:one-kernel}
\end{equation}
The nonconstant local terms are $O(p^{-1-\sigma})$ and
$O(p^{-2\sigma})$. Both prime sums converge for $\sigma>1/2$,
giving the asserted absolute convergence.

In $n=du$, the original restrictions imply $(d,QD_0)=1$.
For such $d$, the conditions remaining on $u$ are
$du\equiv r\pmod Q$ and $(u,D_0)=1$. By M\"{o}bius inversion to $(u,D_0)=1$,
complete multiplicativity of the character gives
\[
 \sum_{\substack{d\leq Z\\(d,QD_0)=1}}b(d)\chi(d)
 \sum_{\ell\mid D_0}\mu(\ell)\chi(\ell)
 \sum_{\substack{v\leq Z/(d\ell)\\d\ell v\equiv r\pmod Q}}\chi(v).
\]
Terms with $(d\ell,q)>1$ vanish, and those with $(\ell,Q)>1$
are empty, since $r$ is a unit modulo $Q$. In every other term,
$d\ell$ is invertible modulo $Q$, so the congruence is
$v\equiv r(d\ell)^{-1}\pmod Q$.
As $(q,Q)=1$, the arguments in any $q$ consecutive terms of this
progression run through every residue modulo $q$. Their sum is zero
by nonprincipality; the last incomplete block has modulus at most $q$.
There are at most $\tau(D_0)$ choices of $\ell$, and by \eqref{arith:one-kernel}
\[
 \sum_{d\leq Z}|b(d)|
 \leq Z^\sigma\sum_{d\geq1}\frac{|b(d)|}{d^\sigma}
 \ll_\sigma Z^\sigma.
\]
Thus the original sum is $O_\sigma(q\tau(D_0)Z^\sigma)$.
This proves the asserted bound. Subtraction of initial
segments gives the interval version.
\end{proof}

\begin{proposition}\label{thm:sf}
Fix $0<\eta\le1$, $C>0$, and $0<a_i<b_i$ for $i=1,2,3$.
Let $X_i\ge2$, $I_i=[a_iX_i,b_iX_i]$, and $X=\max_iX_i$,
and suppose that $\min_iX_i\ge X^\eta$.
All prime factors of the positive integers $A_1,A_2,A_3$ belong to a
finite set $S$, where $2\in S$.
Suppose that $\supp(Q)=S$, $8\mid Q$, and $Q\le(\log X)^C$.
Fix unit residue classes $r_i\bmod Q$ such that the quadratic form
\begin{equation}\label{arith:diagonal-family}
A_1n_1z_1^2+A_2n_2z_2^2-A_3n_3z_3^2
\end{equation}
is soluble over $\mathbb Q_p$ for every $p\in S$ and every positive
triple $n_i\equiv r_i\pmod Q$. All sums below satisfy
\begin{equation}
 n_i\in I_i,\qquad n_i\equiv r_i\pmod Q,\qquad
 \mu^2(n_1n_2n_3)=1.
 \label{arith:squarefree-domain}
\end{equation}
For $p\mid n_1n_2n_3$, put
\[
\eta_p(\bn)=
\begin{cases}
\left(\dfrac{A_2A_3n_2n_3}{p}\right),&p\mid n_1,\\[3pt]
\left(\dfrac{A_1A_3n_1n_3}{p}\right),&p\mid n_2,\\[3pt]
\left(\dfrac{-A_1A_2n_1n_2}{p}\right),&p\mid n_3.
\end{cases}
\]
Then, for every $A>0$, let $f(n)$ be defined in \eqref{f}, we have
\begin{equation}\label{arith:two-endpoint-formula}
\begin{split}
&\sum_{\bn} f(n_1n_2n_3)
\prod_{p\mid n_1n_2n_3}(1+\eta_p(\bn))\\
&\quad=
2\sum_{\bn}f(n_1n_2n_3)
+O_{\eta,C,A,\mathbf a,\mathbf b}
\bigl(X_1X_2X_3(\log X)^{-A}\bigr).
\end{split}
\end{equation}
The implied constant is independent of the values of $A_1,A_2,A_3$.
\end{proposition}
\begin{proof}
Fix $A>0$ and put
\[
L=(\log X)^{6(A+9)}.
\]
Initially suppose that $X$ is sufficiently large, in terms of
$\eta,A,\mathbf a,\mathbf b$, that
\[
2\le L\le X,
\qquad
\min_i\{X_i,a_iX_i\}>L^2.
\]
The conditions in \eqref{arith:squarefree-domain} imply
\[
\mu^2(n_1n_2n_3)=1,
\qquad
(n_1n_2n_3,2QA_1A_2A_3)=1.
\]

Write $\sum^{*}_{\mathbf d,\mathbf e}$ for summation over positive
integer triples satisfying
\[
d_ie_i\in I_i,\qquad
d_ie_i\equiv r_i\pmod Q\quad(1\le i\le3),
\qquad
\mu^2(d_1d_2d_3e_1e_2e_3)=1.
\]
For $n_i=d_ie_i$, put
\[
\Theta(\mathbf d,\mathbf e)
=
\prod_{p\mid d_1d_2d_3}\eta_p(\mathbf n).
\]
Expanding the product and using the multiplicativity of $f$, we obtain
\[
\sum_{\mathbf n}f(n_1n_2n_3)
 \prod_{p\mid n_1n_2n_3}(1+\eta_p(\mathbf n))
=
\sum_{\mathbf d,\mathbf e}^{*}
 \left(\prod_{i=1}^{3}f(d_i)f(e_i)\right)
 \Theta(\mathbf d,\mathbf e).
\]
Quadratic reciprocity
\cite[Chapter~I, \S~3.3]{Serre} gives
\begin{equation}\label{arith:theta}
\begin{aligned}
\Theta(\mathbf d,\mathbf e)
={}&
(-1)^{\frac14\sum_{i<j}(d_i-1)(d_j-1)}
\left(\frac{A_2A_3}{d_1}\right)
\left(\frac{A_1A_3}{d_2}\right)
\left(\frac{-A_1A_2}{d_3}\right)
\prod_{i\ne j}\left(\frac{e_j}{d_i}\right).
\end{aligned}
\end{equation}

The local formula for the Hilbert symbol, together with the
local-solubility hypothesis, yields
\[
\bigl(-A_1A_2n_1n_2,A_1A_3n_1n_3\bigr)_v
=
\begin{cases}
\eta_p(\mathbf n),
 & v=p,\ p\mid n_1n_2n_3,\\
1,&\text{otherwise}.
\end{cases}
\]
Here the symbol at infinity is $1$ because $A_1A_3n_1n_3>0$.
Consequently, by the product formula
\cite[Chapter~III, Theorems~1 and~3]{Serre},
\begin{equation}\label{arith:global-sign}
\prod_{p\mid n_1n_2n_3}\eta_p(\mathbf n)
=
\prod_v
\bigl(-A_1A_2n_1n_2,A_1A_3n_1n_3\bigr)_v
=1.
\end{equation}
Since every $\eta_p(\mathbf n)$ belongs to $\{-1,1\}$,
\begin{equation}\label{arith:complement}
\Theta((1,1,1),\mathbf n)
=
\Theta(\mathbf n,(1,1,1))
=1,
\qquad
\Theta(\mathbf d,\mathbf e)
=
\Theta(\mathbf e,\mathbf d).
\end{equation}
The two endpoints are distinct, since $n_i>L^2$. Hence
\[
\begin{aligned}
&\sum_{\mathbf n}f(n_1n_2n_3)
 \prod_{p\mid n_1n_2n_3}(1+\eta_p(\mathbf n))
 -2\sum_{\mathbf n}f(n_1n_2n_3)\\
&\qquad=
\sum_{\substack{\mathbf d,\mathbf e\\
                 d_1d_2d_3>1,\ e_1e_2e_3>1}}^{*}
 \left(\prod_{i=1}^{3}f(d_i)f(e_i)\right)
 \Theta(\mathbf d,\mathbf e).
\end{aligned}
\]

Consider first the terms for which
\[
d_i>L,\qquad e_j>L
\qquad\text{for some }i\ne j.
\]
Partition the six variables into dyadic intervals, subdividing at $L$:
\[
D_h\le d_h<2D_h,\qquad
E_h\le e_h<2E_h
\qquad(1\le h\le3).
\]
There are $O_{\mathbf b}((\log X)^6)$ boxes. For every nonempty box,
$
\prod_{h=1}^{3}D_hE_h
\ll_{\mathbf b}X_1X_2X_3.$
Choose one linked pair $(d_i,e_j)$ in each such box and fix the
remaining four variables. Then
$D_i,E_j>L/2.$
The restrictions on the two free variables are separate:
\[
\begin{aligned}
d_i&\in[D_i,2D_i)\cap(L,\infty)\cap(I_i/e_i),
&
d_i&\equiv r_ie_i^{-1}\pmod Q,\\
e_j&\in[E_j,2E_j)\cap(L,\infty)\cap(I_j/d_j),
&
e_j&\equiv r_jd_j^{-1}\pmod Q.
\end{aligned}
\]
By \eqref{arith:theta}, their only coupled character factor is
$\left(\frac{e_j}{d_i}\right)$. All other factors, together with
the restrictions involving the fixed variables, can be absorbed into
coefficients of modulus at most $1$. Thus the remaining sum has the form
\[
\sum_{\substack{d_i\le2D_i\\e_j\le2E_j}}
 \alpha_{d_i}\beta_{e_j}\,
 \mu^2(2d_ie_j)
 \left(\frac{e_j}{d_i}\right),
\qquad
|\alpha_{d_i}|,\ |\beta_{e_j}|\le1.
\]
The bilinear estimate of
Friedlander--Iwaniec \cite[Lemma~2]{FriedlanderIwaniec2010}
therefore gives
\[
\begin{aligned}
\left|
\sum_{\substack{d_i\le2D_i\\e_j\le2E_j}}
 \alpha_{d_i}\beta_{e_j}\,
 \mu^2(2d_ie_j)
 \left(\frac{e_j}{d_i}\right)
\right|
&\ll
D_iE_j
 \bigl(D_i^{-1/6}+E_j^{-1/6}\bigr)
 \bigl(\log(8D_iE_j)\bigr)^2\\
&\ll_{\mathbf b}
D_iE_jL^{-1/6}(\log X)^2.
\end{aligned}
\]
Indeed, the squarefree restrictions may be absorbed into the
individual coefficients: for odd $m,n$,
\[
\mu^2(mn)\left(\frac mn\right)
=
\mu^2(m)\mu^2(n)\left(\frac mn\right).
\]
Summing over the four fixed variables and then over the boxes yields
\begin{equation}\label{arith:linked-error}
\left|
\sum_{\substack{\mathbf d,\mathbf e\\
 \exists\,i\ne j:\ d_i>L,\ e_j>L}}^{*}
 \left(\prod_{h=1}^{3}f(d_h)f(e_h)\right)
 \Theta(\mathbf d,\mathbf e)
\right|
\ll_{\mathbf b}
X_1X_2X_3L^{-1/6}(\log X)^8.
\end{equation}

Since $d_ie_i>L^2$ for every $i$,
\[
\left.
\begin{array}{c}
d_i\le L\ \text{or}\ e_j\le L
\quad(i\ne j),\\
d_ie_i>L^2\quad(1\le i\le3)
\end{array}
\right\}
\Longrightarrow
\begin{cases}
\max_i d_i\le L<\min_i e_i,\\
\text{or}\\
\max_i e_i\le L<\min_i d_i.
\end{cases}
\]
By \eqref{arith:complement}, the two contributions are equal.
It therefore suffices to consider
\[
d_i\le L,\qquad d_1d_2d_3>1.
\]

For each such $\mathbf d$, choose $j$ so that
$q=\prod_{i\ne j}d_i>1.$
Then
\[
q\le L^2,\qquad
\mu^2(q)=1,\qquad
(q,2Q)=1,
\]
and $\left(\frac{\cdot}{q}\right)$ is a nonprincipal character
modulo $q$.

Fix $e_i$ for $i\ne j$, and put
\[
D_0=d_1d_2d_3\prod_{i\ne j}e_i
   =d_j\prod_{i\ne j}n_i.
\]
The standard divisor bound gives
\[
D_0\ll_{\mathbf b}LX^2\ll_{\mathbf b}X^3,
\qquad
\tau(D_0)\ll_{\eta,\mathbf b}X^{\eta/12}.
\]
By \eqref{arith:theta}, the dependence on $e_j$ is, up to a factor
of modulus at most $1$,
\[
\mu^2(e_j)f(e_j)
\left(\frac{e_j}{q}\right)
\mathbf 1_{(e_j,D_0)=1}.
\]
Applying Lemma~\ref{lem:small-char} with $\sigma=2/3$, we obtain
\[
\begin{aligned}
\left|
\sum_{\substack{
 e_j\in I_j/d_j\\
 e_j\equiv r_jd_j^{-1}\pmod Q\\
 (e_j,D_0)=1}}
 \mu^2(e_j)f(e_j)
 \left(\frac{e_j}{q}\right)
\right|
&\ll_{\mathbf b}
q\,\tau(D_0)
\left(\frac{X_j}{d_j}\right)^{2/3}\\
&\ll_{\eta,\mathbf b}
L^2X^{\eta/12}
\left(\frac{X_j}{d_j}\right)^{2/3}.
\end{aligned}
\]
For $i\ne j$, the number of available $e_i$ is
$O_{\mathbf b}(X_i/d_i)$. Hence, for each fixed $\mathbf d$,
\[
\begin{aligned}
\left|
\sum_{\mathbf e}^{*}
 \left(\prod_{i=1}^{3}f(d_i)f(e_i)\right)
 \Theta(\mathbf d,\mathbf e)
\right|
&\ll_{\eta,\mathbf b}
L^2X^{\eta/12}
\left(\frac{X_j}{d_j}\right)^{2/3}
\prod_{i\ne j}\frac{X_i}{d_i}\\
&\le
X_1X_2X_3L^2X^{\eta/12}X_j^{-1/3}\\
&\le
X_1X_2X_3L^2X^{-\eta/4}.
\end{aligned}
\]
There are at most $L^3$ possible triples $\mathbf d$. Therefore
\[
\left|
\sum_{\substack{\mathbf d,\mathbf e\\
 \max_i d_i\le L<\min_i e_i\\
 d_1d_2d_3>1}}^{*}
 \left(\prod_{i=1}^{3}f(d_i)f(e_i)\right)
 \Theta(\mathbf d,\mathbf e)
\right|
\ll_{\eta,\mathbf b}
X_1X_2X_3L^5X^{-\eta/4}.
\]

Combining this bound with \eqref{arith:linked-error}, we conclude that
\[
\begin{aligned}
&\left|
\sum_{\mathbf n}f(n_1n_2n_3)
 \prod_{p\mid n_1n_2n_3}(1+\eta_p(\mathbf n))
 -2\sum_{\mathbf n}f(n_1n_2n_3)
\right|\\
&\qquad\ll_{\eta,\mathbf a,\mathbf b}
X_1X_2X_3
\left(
L^{-1/6}(\log X)^8+L^5X^{-\eta/4}
\right)\\
&\qquad=
O_{\eta,A,\mathbf a,\mathbf b}\!\left(
X_1X_2X_3
\left(
(\log X)^{-A-1}
+X^{-\eta/4}(\log X)^{30(A+9)}
\right)\right)\\
&\qquad\ll_{\eta,A,\mathbf a,\mathbf b}
X_1X_2X_3(\log X)^{-A}.
\end{aligned}
\]
All estimates are uniform in $A_1,A_2,A_3,Q$ and the prescribed
residue classes: the $A_i$ occur only in factors of modulus at most
$1$, while $(q,Q)=1$ in the application of
Lemma~\ref{lem:small-char}.

Finally, the remaining bounded range of $X$ is absorbed into the
implied constant, using
\[
0<f(n)\le1,\qquad
\left|
\prod_{p\mid n_1n_2n_3}(1+\eta_p(\mathbf n))-2
\right|
\le 2^{\omega(n_1n_2n_3)}+2
\le n_1n_2n_3+2.
\]
This proves \eqref{arith:two-endpoint-formula}.
\end{proof}

\begin{proposition}\label{lem:ordinary-sf}
Under the interval, congruence, and squarefreeness conditions of
Proposition \ref{thm:sf}, we have
\begin{equation}\label{eq:ordinary-sf}
\begin{split}
\sum_{\bn}f(n_1n_2n_3)
={}\frac{|I_1||I_2||I_3|}{Q^3}
\prod_{p\nmid Q}(1-p^{-1})^3\left(1+\frac3{p+1}\right)
+O_{\mathbf a,\mathbf b}
\bigl(X_1X_2X_3(\min_iX_i)^{-1/3}\bigr).
\end{split}
\end{equation}
The implied constant may be taken independent of $Q$ and of the
chosen unit residue classes.
\end{proposition}

\begin{proof}
By multivariable M\"obius inversion and the multiplicativity
of Dirichlet convolution
\cite[\S\S~4.1, 5.1--5.2]{Toth2014}, there is a multiplicative
function $b:\mathbb N^3\to\mathbb R$ satisfying
\[
\mu^2(n_1n_2n_3)f(n_1n_2n_3)
=
\sum_{\substack{d_1\mid n_1\\d_2\mid n_2\\d_3\mid n_3}}
b(d_1,d_2,d_3),
\]
whose local generating series is
\[
\sum_{\nu_1,\nu_2,\nu_3\ge0}
b(p^{\nu_1},p^{\nu_2},p^{\nu_3})
z_1^{\nu_1}z_2^{\nu_2}z_3^{\nu_3}
=
\left(1+\frac{p}{p+1}(z_1+z_2+z_3)\right)
\prod_{i=1}^3(1-z_i).
\]
The linear coefficients are $-1/(p+1)$; all remaining
nonconstant terms have degree at least two and bounded
coefficients. Consequently,
\begin{align*}
\sum_{\bd\in\mathbb N^3}
\frac{|b(\bd)|}{(d_1d_2d_3)^{2/3}}
=
\prod_p
\left(
\sum_{\nu_1,\nu_2,\nu_3\ge0}
\frac{|b(p^{\nu_1},p^{\nu_2},p^{\nu_3})|}
     {p^{2(\nu_1+\nu_2+\nu_3)/3}}
\right)
=
\prod_p
\left(1+O(p^{-5/3}+p^{-4/3})\right)
\ll1.
\end{align*}

Since $(r_i,Q)=1$, every divisor $d_i\mid n_i$ occurring
in the sum satisfies $(d_i,Q)=1$. Uniformly in such $d_i$,
\[
\#\left\{
m\in\mathbb N:
d_im\in I_i,\ d_im\equiv r_i\pmod Q
\right\}
=
\frac{|I_i|}{Qd_i}+O(1).
\]
Thus
\[
\sum_{\bn}f(n_1n_2n_3)
=
\sum_{\substack{d_i\le b_iX_i\ (1\le i\le3)\\
                (d_1d_2d_3,Q)=1}}
b(\bd)
\prod_{i=1}^3
\left(\frac{|I_i|}{Qd_i}+O(1)\right).
\]
For $d_i\le b_iX_i$, we have $1\le b_iX_i/d_i$, and hence
\[
\prod_{i=1}^3
\left(\frac{|I_i|}{Qd_i}+O(1)\right)
=
\frac{|I_1||I_2||I_3|}{Q^3d_1d_2d_3}
+
O_{\mathbf a,\mathbf b}\left(
\sum_{i=1}^3
\prod_{\substack{1\le j\le3\\j\ne i}}\frac{X_j}{d_j}
\right).
\]
The total error is therefore bounded by
\begin{align*}
\sum_{i=1}^3
\left(\prod_{j\ne i}X_j\right)
\sum_{\substack{\bd\in\mathbb N^3\\
                d_k\le b_kX_k\ (1\le k\le3)}}
\frac{|b(\bd)|}{\prod_{j\ne i}d_j}
&\ll_{\mathbf b}
\sum_{i=1}^3
X_i^{2/3}\left(\prod_{j\ne i}X_j\right)
\sum_{\bd\in\mathbb N^3}
\frac{|b(\bd)|}{(d_1d_2d_3)^{2/3}}\\
&\qquad\ll_{\mathbf b}
X_1X_2X_3\sum_{i=1}^3X_i^{-1/3}
\ll_{\mathbf b}
X_1X_2X_3(\min_iX_i)^{-1/3}.
\end{align*}

The tail of the main term satisfies
\begin{align*}
\sum_{\substack{\bd\in\mathbb N^3\\
                d_i>b_iX_i\text{ for some }i}}
\frac{|b(\bd)|}{d_1d_2d_3}
\le
\sum_{i=1}^3(b_iX_i)^{-1/3}
\sum_{\bd\in\mathbb N^3}
\frac{|b(\bd)|}{(d_1d_2d_3)^{2/3}}
\ll_{\mathbf b}(\min_iX_i)^{-1/3}.
\end{align*}
We may therefore extend the main sum to obtain
\begin{align*}
\sum_{\bn}f(n_1n_2n_3)
={}&
\frac{|I_1||I_2||I_3|}{Q^3}
\sum_{\substack{\bd\in\mathbb N^3\\
                (d_1d_2d_3,Q)=1}}
\frac{b(\bd)}{d_1d_2d_3}
+
O_{\mathbf a,\mathbf b}\left(
X_1X_2X_3(\min_iX_i)^{-1/3}
\right).
\end{align*}
Finally, absolute convergence and the local generating
series give
\[
\sum_{\substack{\bd\in\mathbb N^3\\
                (d_1d_2d_3,Q)=1}}
\frac{b(\bd)}{d_1d_2d_3}
=
\prod_{p\nmid Q}
\left(1-\frac1p\right)^3
\left(1+\frac3{p+1}\right).
\]
This proves \eqref{eq:ordinary-sf}, with an implied constant
independent of $Q$ and the chosen unit residue classes.
\end{proof}

\section{The full bad-prime part}\label{sec:bad}

For a primitive vector $\bx$ with all coordinates nonzero, write
\[
n=|x_0x_1x_2|,\qquad
g=g(\bx)=\gcd(x_0x_1,x_0x_2,x_1x_2),
\]
\[
M=\prod_{p:v_p(n)\ge2}p^{v_p(n)},\qquad
\widehat M=2^{v_2(n)}
\prod_{\substack{p>2\\v_p(n)\ge2}}p^{v_p(n)}.
\]
Note that $\widehat M\le2M$.

\begin{proposition}\label{lem:badmoment}
For $X\ge1$, we have
\begin{equation}\label{eq:badmoment}
\sum_{\substack{\bx\in\Z_{\prim}^3\\0<|x_i|\le X}}
\tau(n)\sqrt{g(\bx)}\,\widehat M^{1/8}
\ll X^3\log^3(2X).
\end{equation}
\end{proposition}
\begin{proof}
The summand is invariant under coordinate sign changes,
so it suffices to consider positive coordinates.
Primitivity gives, for every prime $p$,
\[
\min_i v_p(x_i)=0,
\qquad
v_p(g)=v_p(n)-\max_i v_p(x_i)
\le \frac12v_p(n).
\]
Consequently,
\[
g^2\mid M,
\qquad
\widehat M\le2M,
\qquad
\tau(n)\sqrt g\,\widehat M^{1/8}
\le 2^{1/8}\tau(n)M^{3/8}.
\]

Write
\[
d_i=\prod_{p\mid M}p^{v_p(x_i)},
\qquad
u_i=\frac{x_i}{d_i}
\qquad (i=0,1,2).
\]
Then
\[
d_0d_1d_2=M,
\qquad
(M,u_0u_1u_2)=1,
\qquad
\mu^2(u_0u_1u_2)=1,
\]
and hence
\[
\tau(n)
=
\tau(d_0d_1d_2)\prod_{i=0}^2\tau(u_i).
\]
Here $M$ is squarefull: every prime dividing $M$ occurs
with exponent at least two; the value $M=1$ is included.

Using the classical divisor-sum estimate
$\sum_{u\le Z}\tau(u)\ll Z\log(2Z)$ and dropping the
coprimality and squarefreeness restrictions on the $u_i$, we obtain
\begin{align*}
&\sum_{\substack{1\le x_0,x_1,x_2\le X\\
                 \gcd(x_0,x_1,x_2)=1}}
\tau(n)\sqrt g\,\widehat M^{1/8}\\
&\quad\ll
\sum_{\substack{1\le d_0,d_1,d_2\le X\\
                d_0d_1d_2\ \mathrm{squarefull}}}
\tau(d_0d_1d_2)(d_0d_1d_2)^{3/8}
\prod_{i=0}^2
\left(\sum_{u_i\le X/d_i}\tau(u_i)\right)\\
&\quad\ll
X^3\log^3(2X)
\sum_{\substack{d_0,d_1,d_2\ge1\\
                d_0d_1d_2\ \mathrm{squarefull}}}
\frac{\tau(d_0d_1d_2)}
     {(d_0d_1d_2)^{5/8}}.
\end{align*}

The Euler-product identity for nonnegative multiplicative
functions \cite[\S~5.1]{Toth2014} gives
\begin{align*}
\sum_{\substack{d_0,d_1,d_2\ge1\\
                d_0d_1d_2\ \mathrm{squarefull}}}
\frac{\tau(d_0d_1d_2)}
     {(d_0d_1d_2)^{5/8}}
&=
\prod_p
\left(
1+\sum_{e=2}^{\infty}
(e+1)\binom{e+2}{2}p^{-5e/8}
\right)\\
&=
\prod_p\left(1+O(p^{-5/4})\right)
<\infty.
\end{align*}
Indeed, uniformly for $p\ge2$,
\[
\sum_{e=2}^{\infty}
(e+1)\binom{e+2}{2}p^{-5e/8}
\le
p^{-5/4}
\sum_{e=2}^{\infty}
(e+1)\binom{e+2}{2}2^{-5(e-2)/8}
\ll p^{-5/4}.
\]
Restoring the eight coordinate sign choices proves
\[
\sum_{\substack{\bx\in\Z_{\prim}^3\\0<|x_i|\le X}}
\tau(n)\sqrt{g(\bx)}\,\widehat M^{1/8}
\ll X^3\log^3(2X),
\]
as required.
\end{proof}

\begin{lemma}\label{cor:bad-tail}
Fix $0<a<1$, and let $\SSS(\bx)$ be defined in \eqref{eq:sigma}.
For $X\ge1$ and $R\ge1$, we have
\begin{equation}\label{eq:bad-tail}
\sum_{\substack{\bx\in\Z_{\prim}^3\\
 aX\le|x_i|\le X\ (0\le i\le2)\\
 q_{\bx}\ \mathrm{indefinite},\ \widehat M>R}}\SSS(\bx)
\ll_a X^3\log^3(2X)R^{-1/8}.
\end{equation}
\end{lemma}
\begin{proof}
Let $w_0$ be the fixed weight in Lemma~\ref{prop:conic-input}.
Since $0\le w_0\le\one_{[-1,1]^3}$,
\cite[Lemma~11]{HB} gives, for indefinite $q_{\bx}$,
\[
s(\bx)
\ge s(q_{\bx};w_0)
\gg \|q_{\bx}\|^{-1}
=\frac1{2\|\bx\|_2}
\gg X^{-1}.
\]

By Lemma~\ref{prop:fixed}, including its vanishing assertion,
and \eqref{eq:conic-upper},
\begin{align*}
s(\bx)\SSS(\bx)
\ll
\lim_{Y\to\infty}\frac{N(q_{\bx},Y)}{Y}
\ll
\frac{\tau(4n)\sqrt{4g(\bx)}}{(4n)^{1/3}}
\ll_a X^{-1}\tau(n)\sqrt{g(\bx)},
\end{align*}
where
\[
n=|x_0x_1x_2|\ge a^3X^3,
\qquad
\tau(4n)\le3\tau(n).
\]
Consequently, throughout the prescribed indefinite region,
\[
0\le\SSS(\bx)\ll_a\tau(n)\sqrt{g(\bx)}.
\]

For $R>0$,
\[
\one_{\{\widehat M>R\}}
\le R^{-1/8}\widehat M^{1/8}.
\]
Hence Proposition~\ref{lem:badmoment} yields
\begin{align*}
\sum_{\substack{
\bx\in\Z_{\prim}^3,\ aX\le|x_i|\le X\\
q_{\bx}\ \mathrm{indefinite},\ \widehat M>R}}
\SSS(\bx)
\ll_a
R^{-1/8}
\sum_{\substack{
\bx\in\Z_{\prim}^3\\
0<|x_i|\le X}}
\tau(n)\sqrt{g(\bx)}\,\widehat M^{1/8}
\ll_a
X^3\log^3(2X)R^{-1/8}.
\end{align*}
Restricting to the fixed indefinite region proves
\eqref{eq:bad-tail}.
\end{proof}

\begin{proposition}\label{lem:bad-local}
Write uniquely
\[
x_i=\epsilon_iL_i n_i,\qquad
L_0L_1L_2=\widehat M,\qquad \epsilon_i\in\{-1,1\},
\]
where $L_i$ contains the full prime powers of $x_i$ supported on the
bad primes. Then $n_0n_1n_2$ is odd and squarefree, and is coprime to
$2\widehat M$.
Let
\[
S=\{2\}\cup\supp(\widehat M),\qquad Q=8\widehat M^5.
\]
For each $p\in S$, the density $\sigma_p(\bx)$, which is defined in \eqref{eq:sigma}, depends only on
$\bn\bmod Q$, and
\begin{equation}\label{eq:bad-density-bound}
0\le\prod_{p\in S}\sigma_p(\bx)\le Q.
\end{equation}
\end{proposition}
\begin{proof}
The decomposition is uniquely determined by
\[
\epsilon_i=\operatorname{sgn}(x_i),
\qquad
L_i=\prod_{p\in S}p^{v_p(x_i)},
\qquad
n_i=\frac{|x_i|}{L_i}.
\]
By the definition of $\widehat M$,
\[
L_0L_1L_2=\widehat M,
\qquad
\mu^2(n_0n_1n_2)=1,
\qquad
(n_0n_1n_2,2\widehat M)=1.
\]

Fix $p\in S$ and put
$e_0=\max_i v_p(L_i)+v_p(2).$
For each primitive residue class
$\mathbf a\bmod p^{e_0+1}$, put
$\ell_{\mathbf a}
=
\min_i v_p(2x_i a_i).$
Since $p\nmid\mathbf a$ and $v_p(x_i)=v_p(L_i)$,
$
0\le\ell_{\mathbf a}\le e_0.$
Moreover, $\ell_{\mathbf a}$ is constant on
$\mathbf a+p^{e_0+1}\mathbb Z_p^3$, and
\[
q_{\bx}(\mathbf a+p^{e_0+1}\mathbf z)
\equiv q_{\bx}(\mathbf a)
\pmod{p^{e_0+1+\ell_{\mathbf a}}}.
\]
Thus this ball contributes only if
$p^{e_0+1+\ell_{\mathbf a}}\mid q_{\bx}(\mathbf a)$.
In that case,
\begin{align*}
\frac{q_{\bx}(\mathbf a+p^{e_0+1}\mathbf z)}
     {p^{e_0+1+\ell_{\mathbf a}}}
={}
\frac{q_{\bx}(\mathbf a)}
     {p^{e_0+1+\ell_{\mathbf a}}}
+
\sum_{i=0}^2
\frac{2x_i a_i}{p^{\ell_{\mathbf a}}}z_i
+
p^{e_0+1-\ell_{\mathbf a}}
\sum_{i=0}^2x_i z_i^2.
\end{align*}
This polynomial is affine linear modulo $p$, with at least
one unit linear coefficient. Simple-root Hensel lifting
(cf.~\cite[proof of Lemma~3.2]{Hanke}) therefore gives,
for every $k\ge2e_0+1$,
\[
\begin{aligned}
&\#\left\{
\mathbf y\bmod p^k:
\mathbf y\equiv\mathbf a\pmod{p^{e_0+1}},
\ q_{\bx}(\mathbf y)\equiv0\pmod{p^k}
\right\}\\
&\qquad=
p^{\,2(k-e_0-1-\ell_{\mathbf a})}
p^{\,3\ell_{\mathbf a}}
=
p^{\,2k-2e_0-2+\ell_{\mathbf a}}.
\end{aligned}
\]
Here the first factor counts solutions of the normalized
congruence; the second counts the remaining coordinate digits.

Consequently, for all $k\ge2e_0+1$,
\begin{align*}
p^{-2k}
\#\left\{
\mathbf y\bmod p^k:
p\nmid\mathbf y,\ q_{\bx}(\mathbf y)\equiv0\pmod{p^k}
\right\}
=
\sum_{\substack{
\mathbf a\bmod p^{e_0+1},\ p\nmid\mathbf a\\
p^{e_0+1+\ell_{\mathbf a}}\mid q_{\bx}(\mathbf a)}}
p^{-2e_0-2+\ell_{\mathbf a}}
=
\sigma_p(\bx).
\end{align*}
In particular, the defining limit stabilizes, and
$
0\le\sigma_p(\bx)
\le
p^{3(e_0+1)}p^{-2e_0-2+e_0}
=
p^{2e_0+1}.$
The displayed finite sum depends only on
$\bx\bmod p^{2e_0+1}$, since
\[
\ell_{\mathbf a}\le e_0,
\qquad
e_0+1+\ell_{\mathbf a}\le2e_0+1.
\]

For odd $p\in S$, we have $v_p(\widehat M)\ge2$.
Hence, for every $p\in S$,
\[
\begin{aligned}
2e_0+1
\le 2v_p(\widehat M)+2v_p(2)+1
\le 5v_p(\widehat M)+3v_p(2)
=v_p(Q).
\end{aligned}
\]
For fixed $(\epsilon_i,L_i)$, the residue class
$\bn\bmod Q$ therefore determines each $\sigma_p(\bx)$.
Finally,
\[
0\le
\prod_{p\in S}\sigma_p(\bx)
\le
\prod_{p\in S}p^{v_p(Q)}
=
Q,
\]
which proves \eqref{eq:bad-density-bound}.
\end{proof}

\begin{remark}\label{rem:cores}
Subsequent sums over $\bL$ run only over positive integer triples
satisfying $\gcd(L_0,L_1,L_2)=1$ and such that the total valuation of
each odd prime in $L_0L_1L_2$ is not $1$.
The total valuation of the prime $2$ may be any nonnegative integer.
Together with $n_i\in\Z_{>0}$, the conditions that $n_0n_1n_2$
is odd, squarefree, and coprime to $2L_0L_1L_2$, and the choices
$\epsilon_i\in\{-1,1\}$, the representation
\[
x_i=\epsilon_iL_in_i\qquad(0\le i\le2)
\]
parametrizes all primitive integer coefficient vectors with
nonzero coordinates exactly once.
\end{remark}

\section{Averaging the singular series over the coefficients}\label{sec:arithmean}

\begin{proposition}
\label{lem:local-average}
Let $d\bx$ denote additive Haar measure on $\Z_p^3$, normalized
by $\operatorname{vol}(\Z_p^3)=1$, and let $\sigma_p(\bx)$ be
the primitive local density defined in \eqref{eq:sigma},
extended by zero on the coordinate planes.

For every prime $p$,
\begin{equation}\label{eq:delta}
\delta_p
:=
\int_{\substack{\bx\in\Z_p^3\\p\nmid\bx}}
\sigma_p(\bx)\,d\bx
=
(1-p^{-2})(1-p^{-3}),
\end{equation}
where $p\nmid\bx$ means that at least one coordinate is a
$p$-adic unit.

For every odd prime $p$, define $\delta_{\sfpart,p}$ by
restricting this integral to coefficient vectors for which
$p^2\nmid x_0x_1x_2$. Explicitly,
\begin{equation}\label{eq:delta-sf}
\begin{aligned}
\delta_{\sfpart,p}
:=
\int_{\substack{
\bx\in\Z_p^3\\
v_p(x_0)+v_p(x_1)+v_p(x_2)\le1}}
\sigma_p(\bx)\,d\bx
=
(1-p^{-1})^3
\left[
(1-p^{-2})+\frac3p(1-p^{-1})
\right].
\end{aligned}
\end{equation}
\end{proposition}
\begin{proof}
Set $\sigma_p(\bx)=0$ on the measure-zero set $x_0x_1x_2=0$.
For $k\ge1$, write
\[
\begin{aligned}
\sigma_{p,k}(\bx)
=
p^{-2k}
\#\{\by\bmod p^k:
p\nmid\by,\ q_{\bx}(\by)\equiv0\pmod{p^k}\}=
p^k\int_{\substack{\by\in\Z_p^3\\p\nmid\by}}
\one_{\{p^k\mid q_{\bx}(\by)\}}\,d\by.
\end{aligned}
\]
For each primitive $\by\bmod p^k$, the vector
$(y_0^2,y_1^2,y_2^2)$ is primitive. Hence
\[
\#\{\bx\bmod p^k:
p\nmid\bx,\ q_{\bx}(\by)\equiv0\pmod{p^k}\}
=
p^{2k}-p^{2k-2},
\]
and therefore
\[
\begin{aligned}
\int_{p\nmid\bx}\sigma_{p,k}(\bx)\,d\bx
=
p^{-5k}
(p^{3k}-p^{3k-3})(p^{2k}-p^{2k-2})=(1-p^{-2})(1-p^{-3}).
\end{aligned}
\]

We next justify passage to the density limit.
Fix $k\ge L\ge1$ and $i\in\{0,1,2\}$.
For primitive $\by$, put $t=\min_{j\ne i}v_p(y_j),$
which is finite almost everywhere.
If $t=0$, solve the linear congruence in an $x_j$ with
$j\ne i$ and $p\nmid y_j$.
If $t\ge1$, then $y_i$ is a unit; solving for $x_i$
leaves the condition
$\sum_{j\ne i}x_jy_j^2\equiv0\pmod{p^L}.$
Consequently,
\[
p^k
\int_{\substack{\bx\in\Z_p^3\\p^L\mid x_i}}
\one_{\{p^k\mid q_{\bx}(\by)\}}\,d\bx
=
\begin{cases}
p^{-L},&t=0,\\
p^{-(L-2t)_+},&t\ge1.
\end{cases}
\]
Since $\operatorname{vol}
\{\by\in\Z_p^3:\min_{j\ne i}v_p(y_j)\ge t\}
=p^{-2t},$
Tonelli's theorem gives
\begin{equation}\label{eq:UI}
\begin{aligned}
\int_{\substack{p\nmid\bx\\p^L\mid x_i}}
\sigma_{p,k}(\bx)\,d\bx
&\le
p^{-L}
+
\sum_{1\le t<\lceil L/2\rceil}
p^{-2t}p^{-(L-2t)}
+
p^{-2\lceil L/2\rceil}\\
&\le
\bigl(\lceil L/2\rceil+1\bigr)p^{-L}
\le (L+1)p^{-L}.
\end{aligned}
\end{equation}
This estimate also holds for $p=2$.

The stabilization argument in
Proposition~\ref{lem:bad-local} applies to $p$-adic
coefficients and yields
\[
\sigma_{p,k}(\bx)=\sigma_p(\bx)
\qquad
\left(
\max_i v_p(x_i)<L,\quad k\ge2L+3
\right).
\]
Combining this with \eqref{eq:UI}, we obtain
\[
\begin{aligned}
0
\le
(1-p^{-2})(1-p^{-3})
-
\int_{\substack{p\nmid\bx\\
                 v_p(x_i)<L\ (0\le i\le2)}}
\sigma_p(\bx)\,d\bx
\le 3(L+1)p^{-L}.
\end{aligned}
\]
The integration domains increase to the primitive coefficient
space outside the coordinate planes. Thus monotone convergence
as $L\to\infty$ proves
\[
\delta_p
=
\int_{p\nmid\bx}\sigma_p(\bx)\,d\bx
=
(1-p^{-2})(1-p^{-3}).
\]

Now suppose $p>2$ and $u_0,u_1,u_2\in\Z_p^\times$, where $\mathbb Z_p^\times$ denotes the group of units of
$\mathbb Z_p$.
The finite-field counts and lifting formula in
\cite[Lemma~3.2 and Table~1]{Hanke} give
\[
\sigma_p(u_0,u_1,u_2)=1-p^{-2}.
\]
For $\bx=(pu_0,u_1,u_2)$, the primitive singular
zeros modulo $p$ do not lift modulo $p^2$, since
\[
p\nmid y_0,\quad p\mid y_1,\quad p\mid y_2
\quad\Longrightarrow\quad
q_{\bx}(\by)\equiv pu_0y_0^2\not\equiv0\pmod{p^2}.
\]
The nondegenerate binary form $u_1y_1^2+u_2y_2^2$ over $\FF_p$
has
\[
(p-1)\left(1+\left(\frac{-u_1u_2}{p}\right)\right)
\]
nonzero zeros by \cite[Table~1]{Hanke}.
Each gives $p$ nonsingular ternary zeros, corresponding to the
$p$ choices of $y_0$.
Thus \cite[Lemma~3.2]{Hanke} gives
\[
\begin{aligned}
\sigma_p(pu_0,u_1,u_2)
=
p^{-2}\,p(p-1)
\left(1+\left(\frac{-u_1u_2}{p}\right)\right)=
(1-p^{-1})
\left(1+\left(\frac{-u_1u_2}{p}\right)\right).
\end{aligned}
\]
By quadratic-character orthogonality,
\[
\int_{(\Z_p^\times)^3}
\left(\frac{-u_1u_2}{p}\right)\,d\mathbf u
=0.
\]
Hence, by symmetry in the three coefficients,
\[
\begin{aligned}
\delta_{\sfpart,p}
&=
\int_{(\Z_p^\times)^3}
\sigma_p(\mathbf u)\,d\mathbf u
+
\frac3p
\int_{(\Z_p^\times)^3}
\sigma_p(pu_0,u_1,u_2)\,d\mathbf u\\
&=
(1-p^{-1})^3(1-p^{-2})
+
\frac3p(1-p^{-1})^4\\
&=
(1-p^{-1})^3
\left[
(1-p^{-2})+\frac3p(1-p^{-1})
\right],
\end{aligned}
\]
which proves \eqref{eq:delta-sf}.
\end{proof}

\begin{proposition}\label{thm:arithmean}
Let $\mathcal B\subset\R^3$ be a fixed rectangular box whose closure is bounded away from the three coordinate planes,
and suppose that the whole box lies in a single real indefinite sign region. Let $\SSS(\bx)$ and $\Delta$ be defined in \eqref{eq:sigma} and \eqref{eq:finite-Peyre-product} respectively, we have
\begin{equation}\label{eq:arithmean}
\sum_{\bx\in T\mathcal B\cap\Z_{\prim}^3}\SSS(\bx)
=2\Delta\,\vol(\mathcal B)T^3+o_{\mathcal B}(T^3).
\end{equation}
\end{proposition}
\begin{proof}
By permuting the coordinates and replacing $\bx$ by $-\bx$,
we may assume
\[
\mathcal B
=[a_0,b_0]\times[a_1,b_1]\times[-b_2,-a_2],
\qquad 0<a_i<b_i.
\]
All sums over $\bL$ below are restricted to the admissible
cores of Remark~\ref{rem:cores}, with
$\widehat M(\bL)=L_0L_1L_2$.

Put $R=(\log T)^{80}$. Lemma~\ref{cor:bad-tail} gives
\[
\begin{aligned}
\sum_{\substack{
\bx\in T\mathcal B\cap\Z_{\prim}^3\\
\widehat M(\bx)>R}}
\SSS(\bx)
\ll_{\mathcal B}
T^3\log^3(2T)R^{-1/8}
\ll_{\mathcal B}T^3(\log T)^{-7}.
\end{aligned}
\]

For each retained core, Proposition~\ref{lem:bad-local} gives
\[
\bx=(L_0n_0,L_1n_1,-L_2n_2),
\qquad
\mu^2(n_0n_1n_2)=1,
\qquad
(n_0n_1n_2,2\widehat M)=1.
\]
Use the notation
\[
S=\{2\}\cup\supp(\widehat M),
\qquad Q=8\widehat M^5,
\qquad
X_i=\frac{T}{L_i},
\qquad
I_i=[a_iX_i,b_iX_i],
\qquad X=\max_iX_i.
\]
Uniformly for $\widehat M\le R$, and for sufficiently large $T$,
\[
\frac{T}{R}\le X_i\le T,
\qquad
\min_iX_i\ge X^{1/2},
\qquad
\log X\asymp\log T,
\qquad
Q\le8R^5\le(\log X)^{401}.
\]

For $p\notin S$, the local computations in
Proposition~\ref{lem:local-average}, based on
\cite[Lemma~3.2 and Table~1]{Hanke}, give
\begin{equation}\label{eq:good-density}
\sigma_p(\bx)=
\begin{cases}
1-p^{-2},
   &p\nmid n_0n_1n_2,\\[2mm]
(1-p^{-2})\dfrac{p}{p+1}
   \bigl(1+\eta_p(\bn)\bigr),
   &p\mid n_0n_1n_2,
\end{cases}
\end{equation}
where $\eta_p$ is as in Proposition~\ref{thm:sf},
with the indices relabelled.

Fix $\br\in((\Z/Q\Z)^\times)^3$, where
$
(\Z/Q\Z)^\times
=
\{r\bmod Q:\gcd(r,Q)=1\}.
$
By Proposition~\ref{lem:bad-local},
\[
\prod_{p\in S}\sigma_p(\bx)
=
\prod_{p\in S}\sigma_p(L_0r_0,L_1r_1,-L_2r_2)
\qquad(\bn\equiv\br\pmod Q).
\]
Classes for which this product is zero contribute nothing.
For every other class, local solubility holds throughout
the class at all $p\in S$. Applying
Propositions~\ref{thm:sf} and~\ref{lem:ordinary-sf},
with $\eta=\frac12$, $C=401$, and $A=4000$, therefore yields
\[
\begin{aligned}
\sum_{\substack{
n_i\in I_i,\ n_i\equiv r_i\pmod Q\\
\mu^2(n_0n_1n_2)=1}}
f(n_0n_1n_2)
\prod_{p\mid n_0n_1n_2}
\bigl(1+\eta_p(\bn)\bigr)
=&
\frac{2T^3\vol(\mathcal B)}{\widehat M Q^3}
\prod_{p\notin S}
(1-p^{-1})^3\left(1+\frac3{p+1}\right)\\
&+
O_{\mathcal B}\!\left(
\frac{T^3}{\widehat M}
\left[
(\log T)^{-4000}
+\left(\frac RT\right)^{1/3}
\right]\right).
\end{aligned}
\]
In particular, the factor $2$ comes from
Proposition~\ref{thm:sf}.

Since
\[
(1-p^{-2})(1-p^{-1})^3
\left(1+\frac3{p+1}\right)
=\delta_{\sfpart,p},
\]
define
\begin{equation}\label{eq:core-constant}
c_{\bL}
=
\frac{1}{\widehat M Q^3}
\sum_{\br\in((\Z/Q\Z)^\times)^3}
\prod_{p\in S}
\sigma_p(L_0r_0,L_1r_1,-L_2r_2)
\prod_{p\notin S}\delta_{\sfpart,p}.
\end{equation}
These constants are nonnegative.

The error estimates are uniform in $\bL$ and $\br$.
Indeed, Proposition~\ref{lem:bad-local} gives
\[
\begin{aligned}
\sum_{\widehat M(\bL)\le R}
\sum_{\br\in((\Z/Q\Z)^\times)^3}
\prod_{p\in S}
\sigma_p(L_0r_0,L_1r_1,-L_2r_2)
\le R^3(8R^5)^4\ll R^{23}.
\end{aligned}
\]
Consequently, the total error from the retained cores is
\[
\begin{aligned}
\ll_{\mathcal B}
T^3R^{23}
\left[
(\log T)^{-4000}
+\left(\frac RT\right)^{1/3}
\right]
\ll_{\mathcal B}
T^3(\log T)^{-2160}
+
T^{8/3}(\log T)^{5600/3}
=o_{\mathcal B}(T^3).
\end{aligned}
\]
Combining this with the discarded-core estimate gives
\begin{equation}\label{eq:core-sum}
\sum_{\bx\in T\mathcal B\cap\Z_{\prim}^3}\SSS(\bx)
=
2T^3\vol(\mathcal B)
\sum_{\widehat M(\bL)\le R}c_{\bL}
+o_{\mathcal B}(T^3).
\end{equation}

It remains to identify the sum of the constants.
The Chinese remainder theorem, local constancy from
Proposition~\ref{lem:bad-local}, and changes of variables
by $p$-adic units give
\[
c_{\bL}
=
\prod_{p\in S}
\left(
\int_{\substack{
\mathbf z\in\Z_p^3\\
v_p(z_i)=v_p(L_i)\ (0\le i\le2)}}
\sigma_p(\mathbf z)\,d\mathbf z
\right)
\prod_{p\notin S}\delta_{\sfpart,p}.
\]
Here the valuation-stratum Jacobians satisfy
\[
\prod_{p\in S}p^{-\sum_i v_p(L_i)}
=\frac1{L_0L_1L_2}
=\frac1{\widehat M}.
\]
Thus the integrals use additive Haar measure, without
normalization by the measure of the unit group.

By Proposition~\ref{lem:local-average},
\[
\delta_{\sfpart,p}
+
\int_{\substack{
\mathbf z\in\Z_p^3,\ p\nmid\mathbf z\\
\sum_i v_p(z_i)\ge2}}
\sigma_p(\mathbf z)\,d\mathbf z
=\delta_p
\qquad(p>2),
\]
whereas at $p=2$ all primitive valuation strata contribute:
$
\int_{\substack{\mathbf z\in\Z_2^3\\2\nmid\mathbf z}}
\sigma_2(\mathbf z)\,d\mathbf z=\delta_2.$
The admissibility conditions on $\bL$ therefore imply,
by Tonelli's theorem, that for every $z\ge2$,
\begin{equation}\label{eq:tonelli-core}
\sum_{\substack{
\bL\ \mathrm{admissible}\\
p\mid\widehat M(\bL)\ \Rightarrow\ p\le z}}
c_{\bL}
=
\prod_{p\le z}\delta_p
\prod_{p>z}\delta_{\sfpart,p}.
\end{equation}

Finally,
\[
\delta_p=(1-p^{-2})(1-p^{-3}),
\qquad
\delta_{\sfpart,p}=1+O(p^{-2}),
\qquad
\delta_{\sfpart,p}>0\quad(p>2).
\]
Hence, by monotone convergence on the left-hand side
of \eqref{eq:tonelli-core},
\[
\begin{aligned}
\sum_{\bL\ \mathrm{admissible}}c_{\bL}
=
\lim_{z\to\infty}
\left(\prod_{p\le z}\delta_p\right)
\left(\prod_{p>z}\delta_{\sfpart,p}\right)=
\prod_p(1-p^{-2})(1-p^{-3})
=\frac1{\zeta(2)\zeta(3)}
=\Delta.
\end{aligned}
\]
Since $c_{\bL}\ge0$ and every core has finite
$\widehat M(\bL)$,
\[
\sum_{\widehat M(\bL)\le R}c_{\bL}
=\Delta+o(1)
\qquad(R\to\infty).
\]
Substituting into \eqref{eq:core-sum} proves
\[
\sum_{\bx\in T\mathcal B\cap\Z_{\prim}^3}\SSS(\bx)
=
2\Delta\,\vol(\mathcal B)T^3
+o_{\mathcal B}(T^3),
\]
as required.
\end{proof}

\section{Uniform upper bounds for height intervals and small coordinates}\label{sec:shape}

We first recall the parametrization and the coprimality conditions
needed for the uniform estimates.
By \cite[Lemma 5]{LB}, integer solutions that are primitive in both sets of variables and have all six coordinates nonzero can be parametrized by
$f_i,g_i,h_i\in\Z_{>0}$ and $u_i,v_i\in\Z\setminus\{0\}$ as
\begin{equation}\label{eq:param}
x_i=f_jf_kg_j^2g_k^2h_i^2u_i,\qquad
y_i=f_ig_ih_jh_kv_i,\qquad \{i,j,k\}=\{0,1,2\},
\end{equation}
where
\[
f_0u_0v_0^2+f_1u_1v_1^2+f_2u_2v_2^2=0,
\]
and, for $i\ne j$,
\begin{equation}\label{eq:param-coprime}
\begin{gathered}
(f_i,f_jg_jh_iu_iv_j)=1,\qquad
(g_i,g_jh_iu_iv_iv_j)=1,\\
(h_i,h_jv_i)=1,\qquad
(u_i,u_j)=(v_i,v_j)=1.
\end{gathered}
\end{equation}
No squarefree condition is imposed on $f_i$.
The inverse construction proceeds by setting
\[
h_i=\gcd(|y_j|,|y_k|),\quad
y_i=h_jh_ky_i',\quad x_i=h_i^2x_i',
\]
\[
X_i=\gcd(|x_j'|,|x_k'|),\quad x_i'=X_jX_ku_i,
\quad X_i=f_ig_i^2,\quad y_i'=f_ig_iv_i,\quad(g_i,v_i)=1.
\]
Uniqueness in the last step can be checked prime by prime:
if $e=v_p(X_i)$ and $t=v_p(y_i')$, the equation gives $e\le2t$, and
\[
v_p(g_i)=\max(0,e-t),\quad
v_p(f_i)=e-2\max(0,e-t),\quad
v_p(v_i)=t-e+\max(0,e-t).
\]
The full equivalence between the coprimality conditions and primitivity is given in the cited lemma.

The conditions \eqref{eq:param-coprime}, together with the equation, also imply
$(u_iv_i,u_jv_j)=1$.
For example, if $p\mid u_i$ and $p\mid v_j$,
and $k$ is the remaining index, then $u_k,v_k$ are units.
The equation forces $p\mid f_k$, contradicting $(f_k,v_j)=1$.
We may therefore use \cite[Lemma 4]{LB}:
when the $f_i$ are pairwise coprime, the number of solutions in the corresponding $\mathbf u,\mathbf v$ box satisfies
\begin{equation}\label{eq:LBbox}
N_{\mathbf f}(\mathbf U,\mathbf V)
\ll_\eps
|f_0f_1f_2|^\eps
(U_0U_1U_2)^{2/3}(V_0V_1V_2)^{1/3}
M_\eps(\mathbf U),
\end{equation}
where
\[
M_\eps(\mathbf U)=
1+\max_{\{i,j,k\}=\{0,1,2\}}
(U_iU_j)^{-1/2+\eps}\log(2U_k).
\]

\begin{proposition}\label{lem:shape}
Fix $0<\gamma<1/6$.
For $B\ge2$, $1\le R_1<R_2$, and $0<a\le1$, we have
\begin{equation}\label{eq:shape}
\begin{split}
&\#\{([\bx],[\by])\in U(\Q):H([\bx],[\by])\le B,\ R_1<|\bx|\le R_2,\
\min_i|x_i|\le a|\bx|\}\\
&\qquad\ll_\gamma
Ba^{2/3}\bigl(1+\log(R_2/R_1)\bigr)+BR_1^{-\gamma}.
\end{split}
\end{equation}
Without the shape restriction, take $a=1$.
\end{proposition}
\begin{proof}
Use the parametrization of \cite[Lemma~5]{LB}, recalled in
\eqref{eq:param}--\eqref{eq:param-coprime}.
Partition the parameters into
\[
f_i\in[F_i,2F_i),\quad
g_i\in[G_i,2G_i),\quad
h_i\in[H_i,2H_i),\quad
|u_i|\in[U_i,2U_i),\quad
|v_i|\in[V_i,2V_i),
\]
where all uppercase variables are powers of $2$ and are at least
$1$. All sums over these variables below are dyadic.

Put
\[
A_i=F_jF_kG_j^2G_k^2H_i^2
\quad(\{i,j,k\}=\{0,1,2\}),
\qquad
M=\max_i A_iU_i.
\]
Then
\[
A_iU_i\le |x_i|<2^9A_iU_i,
\qquad
M\le|\bx|<2^9M.
\]
Thus every box meeting the required region satisfies
\[
2^{-9}R_1<M\le R_2,
\qquad
\min_iA_iU_i\le2^9aM,
\]
and
\begin{equation}\label{eq:dyadic-height}
M^2\max_{\{i,j,k\}=\{0,1,2\}}
F_iG_iH_jH_kV_i\le B.
\end{equation}

Apply \cite[Lemma~4]{LB} with $\eps=1/6$.
The coprimality hypotheses were verified immediately before
\eqref{eq:LBbox}. Summing over the $f_i,g_i,h_i$ in their
respective intervals, each box contributes at most
\[
\ll
\prod_{i=0}^2
F_i^{7/6}G_iH_iU_i^{2/3}V_i^{1/3}
\left(
1+\sum_{k=0}^2
\frac{\log(2U_k)}
     {\bigl(\prod_{i\ne k}U_i\bigr)^{1/3}}
\right).
\]
Moreover, \eqref{eq:dyadic-height} gives
\[
\sum_{\substack{
V_0,V_1,V_2\\
F_iG_iH_jH_kV_i\le B/M^2\ (0\le i\le2)}}
(V_0V_1V_2)^{1/3}
\ll
\frac{B}
{M^2\prod_{i=0}^2(F_iG_iH_i^2)^{1/3}}.
\]
Consequently, the count in \eqref{eq:shape} is bounded by
an absolute constant times
\begin{equation}\label{eq:M1}
\begin{split}
B\sum_{\substack{
\mathbf F,\mathbf G,\mathbf H,\mathbf U\\
2^{-9}R_1<M\le R_2\\
\min_iA_iU_i\le2^9aM}}
\frac{\prod_{i=0}^2
F_i^{5/6}G_i^{2/3}H_i^{1/3}U_i^{2/3}}
{M^2}
\left(
1+\sum_{k=0}^2
\frac{\log(2U_k)}
     {\bigl(\prod_{i\ne k}U_i\bigr)^{1/3}}
\right).
\end{split}
\end{equation}

First consider the contribution of $1$.
Partition according to a largest $A_iU_i$; by symmetry,
take $M=A_0U_0$.
If $2^9a<1$, the shape condition forces
$A_1U_1\le2^9aM$ or $A_2U_2\le2^9aM$.
If $2^9a\ge1$, use the unrestricted sum.
In both cases, geometric summation gives
\[
\begin{aligned}
\sum_{\substack{
U_1\le M/A_1,\ U_2\le M/A_2\\
\min\{M,A_1U_1,A_2U_2\}\le2^9aM}}
(U_1U_2)^{2/3}
\ll
\min\{1,2^9a\}^{2/3}
\left(\frac{M^2}{A_1A_2}\right)^{2/3}
\ll
a^{2/3}
\left(\frac{M^2}{A_1A_2}\right)^{2/3}.
\end{aligned}
\]
Since
\[
A_0A_1A_2
=
\prod_{i=0}^2F_i^2G_i^4H_i^2,
\qquad M=A_0U_0,
\]
the contribution for fixed $\mathbf F,\mathbf G,\mathbf H,U_0$
is
\[
\begin{aligned}
\ll
Ba^{2/3}
\frac{
\prod_{i=0}^2F_i^{5/6}G_i^{2/3}H_i^{1/3}
\ U_0^{2/3}}
{M^2}
\left(\frac{M^2}{A_1A_2}\right)^{2/3}
=
Ba^{2/3}
\prod_{i=0}^2F_i^{-1/2}G_i^{-2}H_i^{-1}.
\end{aligned}
\]
For fixed $\mathbf F,\mathbf G,\mathbf H$,
\[
\#\{U_0:2^{-9}R_1<A_0U_0\le R_2\}
\ll1+\log(R_2/R_1).
\]
Hence the total contribution of $1$ is
\begin{equation}\label{eq:M1summed}
\begin{aligned}
\ll
Ba^{2/3}\bigl(1+\log(R_2/R_1)\bigr)
\sum_{\mathbf F,\mathbf G,\mathbf H}
\prod_{i=0}^2F_i^{-1/2}G_i^{-2}H_i^{-1}\ll
Ba^{2/3}\bigl(1+\log(R_2/R_1)\bigr).
\end{aligned}
\end{equation}

For the other three summands in \eqref{eq:M1},
symmetry reduces the argument to $k=0$.
No ordering of the $U_i$ is required. Using
\[
M^2\ge
(A_0U_0)(A_1U_1)^{1/2}(A_2U_2)^{1/2},
\]
the corresponding summand is at most
\begin{equation}\label{eq:M2}
\begin{aligned}
&B\log(2U_0)
\frac{
\prod_{i=0}^2F_i^{5/6}G_i^{2/3}H_i^{1/3}
\ U_0^{2/3}(U_1U_2)^{1/3}}
{M^2}\\
&\quad\le
B\log(2U_0)\,
F_0^{-1/6}(F_1F_2)^{-2/3}
G_0^{-4/3}(G_1G_2)^{-7/3}\\
&\qquad\qquad{}\times
H_0^{-5/3}(H_1H_2)^{-2/3}
U_0^{-1/3}(U_1U_2)^{-1/6}.
\end{aligned}
\end{equation}
All dyadic scales are at least $1$, so
$M\le\prod_{i=0}^2F_iG_i^2H_i^2U_i.$
Together with $M>2^{-9}R_1$, this bounds the last expression by
\[
\ll_\gamma
BR_1^{-\gamma}\log(2U_0)
\prod_{i=0}^2
F_i^{\gamma-1/6}
G_i^{2\gamma-4/3}
H_i^{2\gamma-2/3}
U_i^{\gamma-1/6}.
\]
Since $0<\gamma<1/6$, every exponent is negative, and
\[
\sum_{\mathbf F,\mathbf G,\mathbf H,\mathbf U}
\log(2U_0)
\prod_{i=0}^2
F_i^{\gamma-1/6}
G_i^{2\gamma-4/3}
H_i^{2\gamma-2/3}
U_i^{\gamma-1/6}
\ll_\gamma1.
\]
Thus the three remaining contributions together are $\ll_\gamma BR_1^{-\gamma}.$
Combining this with \eqref{eq:M1summed} proves
\[
\begin{aligned}
&\#\bigl\{
P\in U(\Q):
H(P)\le B,\ R_1<|\bx|\le R_2,\
\min_i|x_i|\le a|\bx|
\bigr\}\\
&\qquad\ll_\gamma
Ba^{2/3}\bigl(1+\log(R_2/R_1)\bigr)
+BR_1^{-\gamma}.
\end{aligned}
\]
Taking $a=1$ removes the shape restriction.
\end{proof}

\begin{lemma}\label{cor:kappa-upper}
Let $\kappa_{[\bx]}$ be defined in \eqref{eq:kappa}, and put
$K(T)=\sum_{|\bx|\le T}\kappa_{[\bx]}$, where each projective
coefficient point with nonzero coordinates is counted once using
a primitive integral representative.
For $T\ge1$, $0<a\le1$, and $0<\gamma<1/6$, we have
\begin{align}
\sum_{|\bx|\le T}\frac{\kappa_{[\bx]}}{|\bx|^2}
&\ll\log(2T),\qquad K(T)\ll T^2,\label{eq:Kupper}\\
\sum_{\substack{|\bx|\le T\\\min|x_i|\le a|\bx|}}\kappa_{[\bx]}
&\ll_\gamma a^{2/3}T^2+T^{2-\gamma}+1.\label{eq:Kshape}
\end{align}
\end{lemma}
\begin{proof}
All sums run over projective coefficient points with nonzero
coordinates, counted once using primitive integral representatives.
By \eqref{eq:kappa},
\[
\kappa_{[\bx]}\ge0,
\qquad
\lim_{B\to\infty}
\frac1B
C_{[\bx]}\!\left(\frac{B}{|\bx|^2}\right)
=
\frac{\kappa_{[\bx]}}{|\bx|^2}.
\]
For fixed $1\le R_1<R_2$ and $0<a\le1$, the coefficient
sum below is finite. Since
$H([\bx],[\by])=|\bx|^2|\by|$, for $0<\gamma<\frac16,$ Proposition~\ref{lem:shape}
therefore gives
\begin{equation}\label{eq:kappa-slab}
\sum_{\substack{
R_1<|\bx|\le R_2\\
\min_i|x_i|\le a|\bx|}}
\frac{\kappa_{[\bx]}}{|\bx|^2}=
\lim_{B\to\infty}\frac1B
\sum_{\substack{
R_1<|\bx|\le R_2\\
\min_i|x_i|\le a|\bx|}}
C_{[\bx]}\!\left(\frac{B}{|\bx|^2}\right)
\ll_\gamma
a^{2/3}\bigl(1+\log(R_2/R_1)\bigr)
+R_1^{-\gamma}.
\end{equation}

Taking $R_1=2^j$ and $R_2=2^{j+1}$, we obtain, for $j\ge0$,
\[
\sum_{\substack{
2^j<|\bx|\le2^{j+1}\\
\min_i|x_i|\le a|\bx|}}
\frac{\kappa_{[\bx]}}{|\bx|^2}
\ll_\gamma a^{2/3}+2^{-\gamma j},
\]
and hence
\[
\begin{aligned}
\sum_{\substack{
2^j<|\bx|\le2^{j+1}\\
\min_i|x_i|\le a|\bx|}}
\kappa_{[\bx]}
\le
2^{2j+2}
\sum_{\substack{
2^j<|\bx|\le2^{j+1}\\
\min_i|x_i|\le a|\bx|}}
\frac{\kappa_{[\bx]}}{|\bx|^2}
\ll_\gamma
a^{2/3}2^{2j}+2^{(2-\gamma)j}.
\end{aligned}
\]

The finitely many coefficient points of height $1$ satisfy
$
\sum_{|\bx|=1}\kappa_{[\bx]}\ll1.$
Taking $a=1$ and fixing $\gamma=1/12$, nonnegativity and
dyadic summation give, for $T\ge1$,
\[
\begin{aligned}
\sum_{|\bx|\le T}
\frac{\kappa_{[\bx]}}{|\bx|^2}
\ll
1+\sum_{\substack{j\ge0\\2^j<T}}1
\ll\log(2T),\quad\text{and}\quad
K(T)
\ll
1+\sum_{\substack{j\ge0\\2^j<T}}2^{2j}
\ll T^2.
\end{aligned}
\]
This proves \eqref{eq:Kupper}.

Finally, for $0<a\le1$ and $0<\gamma<1/6$,
\[
\begin{aligned}
\sum_{\substack{
|\bx|\le T\\
\min_i|x_i|\le a|\bx|}}
\kappa_{[\bx]}
\ll_\gamma
1+
\sum_{\substack{j\ge0\\2^j<T}}
\left(a^{2/3}2^{2j}+2^{(2-\gamma)j}\right)
\ll_\gamma
1+a^{2/3}T^2+T^{2-\gamma},
\end{aligned}
\]
which is \eqref{eq:Kshape}.
\end{proof}

\section{Averaging the leading coefficients of the fibers}\label{sec:K}

\begin{proposition}\label{thm:K}
Let $K(T)$ and $\cP$ be as in Lemma~\ref{cor:kappa-upper}
and Theorem~\ref{thm:main}, respectively. Then, as $T\to\infty$,
\begin{equation}\label{eq:K}
K(T)=\cP T^2+o(T^2).
\end{equation}
\end{proposition}
\begin{proof}
By \eqref{eq:kappa}, each projective coefficient point has two
primitive integral representatives, and therefore
\begin{equation}\label{eq:Kraw}
K(T)
=
\frac18
\sum_{\substack{
\bx\in\Z_{\prim}^3\\
0<|x_i|\le T\ (0\le i\le2)}}
s(\bx)\SSS(\bx).
\end{equation}

Fix $0<\varepsilon<1$ and put
$\mathcal D_\varepsilon
=
\{\bt\in[-1,1]^3:\min_i|t_i|\ge\varepsilon\}.$
This is the union of eight rectangular boxes.
On each indefinite component, continuity of $s$ follows,
after a permutation and an overall sign change, from
\[
s(t_0,t_1,-t_2)
=
\int_{[-1,1]^2}
\frac{
\one_{\{t_0u^2+t_1v^2\le t_2\}}
}{
\sqrt{t_2(t_0u^2+t_1v^2)}
}\,du\,dv,
\qquad \varepsilon\le t_i\le1.
\]
Indeed, the integrand is bounded almost everywhere by
$
\frac1{\varepsilon\sqrt{u^2+v^2}}
\in L^1([-1,1]^2),$
and the moving boundary
$t_0u^2+t_1v^2=t_2$ has planar measure zero.
Dominated convergence gives continuity; on the definite
components, $s=0$.

Since $\SSS(\bx)\ge0$, Proposition~\ref{thm:arithmean}
and upper and lower rectangular step-function approximations
now give
\begin{equation}\label{eq:weak}
\frac1{8T^3}
\sum_{\substack{
\bx\in\Z_{\prim}^3\\
\bx/T\in\mathcal D_\varepsilon}}
s(\bx/T)\SSS(\bx)
\longrightarrow
\frac{\Delta}{4}
\int_{\mathcal D_\varepsilon}s(\bt)\,d\bt.
\end{equation}
Here partition boundaries contribute $o(T^3)$:
enclose them in finitely many thin rectangles, apply
Proposition~\ref{thm:arithmean}, and let their thickness
tend to zero.

It remains to remove the coordinate cutoff.
The implication
\[
\min_i|x_i|<\varepsilon T
\quad\Longrightarrow\quad
|\bx|\le\varepsilon^{1/4}T
\quad\text{or}\quad
\min_i|x_i|<\varepsilon^{3/4}|\bx|
\]
and Lemma~\ref{cor:kappa-upper}, with $\gamma=1/12$, give
\[
\begin{aligned}
\frac1{T^2}
\sum_{\substack{
|\bx|\le T\\
\min_i|x_i|<\varepsilon T}}
\kappa_{[\bx]}
&\le
\frac{K(\varepsilon^{1/4}T)}{T^2}
+
\frac1{T^2}
\sum_{\substack{
|\bx|\le T\\
\min_i|x_i|\le\varepsilon^{3/4}|\bx|}}
\kappa_{[\bx]}\\
&\ll
\varepsilon^{1/2}+T^{-1/12}+T^{-2},
\end{aligned}
\]
for fixed $\varepsilon$ and sufficiently large $T$.
The implied constant is independent of $\varepsilon$.

Using $s(\bx/T)=Ts(\bx)$ in \eqref{eq:Kraw}, we obtain
\[
\begin{aligned}
0
\le
\frac{K(T)}{T^2}
-
\frac1{8T^3}
\sum_{\substack{
\bx\in\Z_{\prim}^3\\
\bx/T\in\mathcal D_\varepsilon}}
s(\bx/T)\SSS(\bx)
=
\frac1{T^2}
\sum_{\substack{
|\bx|\le T\\
\min_i|x_i|<\varepsilon T}}
\kappa_{[\bx]}
\ll
\varepsilon^{1/2}+T^{-1/12}+T^{-2}.
\end{aligned}
\]
Consequently, \eqref{eq:weak} implies
\[
\limsup_{T\to\infty}
\left|
\frac{K(T)}{T^2}
-
\frac{\Delta}{4}
\int_{\mathcal D_\varepsilon}s(\bt)\,d\bt
\right|
\ll\varepsilon^{1/2}.
\]

By Proposition~\ref{lem:I}, $s$ is nonnegative and integrable
on $[-1,1]^3$, and
\[
\lim_{\varepsilon\downarrow0}
\int_{\mathcal D_\varepsilon}s(\bt)\,d\bt
=
\int_{[-1,1]^3}s(\bt)\,d\bt
=
I
=
48\pi-56.
\]
Letting $\varepsilon\downarrow0$ and using
\eqref{eq:finite-Peyre-product}, we conclude that
\[
\lim_{T\to\infty}\frac{K(T)}{T^2}
=
\frac{\Delta I}{4}
=
\frac{12\pi-14}{\zeta(2)\zeta(3)}
=
\cP.
\]
This proves \eqref{eq:K}.
\end{proof}

\section{Summing conic counting errors over the full height range}\label{sec:error}

\subsection{Mean absolute error in balanced boxes}\label{sec:avg-error}

\begin{proposition}\label{lem:avg-error}
Fix $0<a<1$, $\varepsilon>0$, and a real-valued weight
$w\in C_c^\infty(\mathbb R^3)$. With the notation of \eqref{PV}, put
\[
P_{\bx,w}(Y)=P_{q_{\bx},w}(Y),
\qquad
\nu_{\bx,w}
=\frac12s(q_{\bx};w)\SSS(\bx).
\]
For $Y\ge X\ge2$ and any subset
\[
\mathcal A\subseteq
\left\{
\bx\in\mathbb Z_{\mathrm{prim}}^3:
aX\le |x_i|\le X\quad(0\le i\le2)
\right\},
\]
we have
\begin{equation}\label{eq:avg-error}
\sum_{\bx\in\mathcal A}
\left|P_{\bx,w}(Y)-\nu_{\bx,w}Y\right|
\ll_{a,w,\varepsilon}
X^{2+\varepsilon}Y
\left(\frac XY\right)^{1/4}\log(2Y).
\end{equation}
More generally, for fixed $0<c_-\le c_+$ and arbitrary parameters
$c_-Y\le Y_{\bx}\le c_+Y$,
\[
\sum_{\bx\in\mathcal A}
\left|P_{\bx,w}(Y_{\bx})-\nu_{\bx,w}Y_{\bx}\right|
\ll_{a,w,\varepsilon,c_-,c_+}
X^{2+\varepsilon}Y
\left(\frac XY\right)^{1/4}\log(2Y).
\]
The implied constants are independent of $\mathcal A$ and of the
individual choices of $Y_{\bx}$.
\end{proposition}

\begin{proof}
We prove the assertion for varying $Y_{\bx}$; taking
$c_-=c_+=1$ gives \eqref{eq:avg-error}.
Initially suppose that $c_-Y\ge2$.

Fibers without rational points contribute zero by
Lemma~\ref{prop:fixed} and may be deleted from $\mathcal A$.
For each remaining fiber, put
\[
\ell_{\bx}
=
\min_{\substack{
\by\in\mathbb Z_{\mathrm{prim}}^3\\
q_{\bx}(\by)=0}}
\|\by\|_2.
\]
The balanced condition implies
\[
|D_{q_{\bx}}|\asymp_a X^3,
\qquad
\|q_{\bx}\|\asymp_a X,
\qquad
\rho(q_{\bx})\ll_a1,
\]
and
\[
4^{\omega(|D_{q_{\bx}}|)}
\tau(|D_{q_{\bx}}|)
\ll_{\varepsilon}X^\varepsilon.
\]
By \eqref{eq:conic-upper},
\[
N(q_{\bx},Z)
\ll_a
\tau(|D_{q_{\bx}}|)
\left(1+\frac{Z\sqrt{g(\bx)}}{X}\right)
\qquad(Z\ge1).
\]
Using the fixed-fiber limit for $w_0$, we obtain
\[
\begin{aligned}
0\le\nu_{\bx,w_0}
=
\lim_{Z\to\infty}\frac{P_{\bx,w_0}(Z)}{Z}
\le
\limsup_{Z\to\infty}\frac{N(q_{\bx},Z)}{Z}
\ll_a
\frac{\tau(|D_{q_{\bx}}|)\sqrt{g(\bx)}}{X}.
\end{aligned}
\]

Suppose first that $w$ is supported in the unit Euclidean ball
and $s(q_{\bx};w)>0$. Since
$
\nu_{\bx,w}
\frac{s(q_{\bx};w_0)}{s(q_{\bx};w)}
=
\nu_{\bx,w_0},
$
from\eqref{eq:HB-input}, we have
\begin{equation}\label{eq:pointwise-error}
\begin{aligned}
\left|P_{\bx,w}(Y_{\bx})-\nu_{\bx,w}Y_{\bx}\right|
\ll_{a,w,\varepsilon}{}&
X^{-1/2+\varepsilon}
\sqrt{g(\bx)}\,
Y_{\bx}^{3/4}\ell_{\bx}^{-1/4}
\log(2Y_{\bx})
+X^\varepsilon.
\end{aligned}
\end{equation}

For general $w$, choose a fixed $c\ge1$ such that
\[
\widetilde w(\mathbf u)=w(c\mathbf u),
\qquad
\operatorname{supp}(\widetilde w)
\subset\{\mathbf u:\|\mathbf u\|_2<1\}.
\]
Then
\[
P_{\bx,w}(Y_{\bx})
=
P_{\bx,\widetilde w}(cY_{\bx}),
\qquad
\nu_{\bx,\widetilde w}=c^{-1}\nu_{\bx,w}.
\]
Choose $M>0$, depending only on $w$, such that
$|\widetilde w|\le Mw_0$. The decomposition
$
\widetilde w=(\widetilde w+2Mw_0)-2Mw_0$
satisfies
\[
s(q_{\bx};\widetilde w+2Mw_0)
\ge M s(q_{\bx};w_0)>0,
\qquad
s(q_{\bx};2Mw_0)>0.
\]
Applying the preceding estimate to these two weights and subtracting
proves \eqref{eq:pointwise-error} for arbitrary real-valued
$w\in C_c^\infty(\mathbb R^3)$.

By \cite[Theorem~2]{HB}, we have
$
1\le\ell_{\bx}\ll\|q_{\bx}\|\ll X.$
For primitive $\by$, the vector $(y_0^2,y_1^2,y_2^2)$ is primitive.
Thus Lemma~\ref{prop:linear} yields, for $1\le L\ll X$,
\[
\begin{aligned}
\#\{\bx\in\mathcal A:\ell_{\bx}\le L\}
&\le
\sum_{\substack{
\by\in\mathbb Z_{\mathrm{prim}}^3\\
|\by|\le L}}
\#\{\bx\in\mathcal A:q_{\bx}(\by)=0\}\\
&\ll
\sum_{1\le r\le L}
r^2\left(\frac{X^2}{r^2}+1\right)
\ll X^2L+L^3
\ll X^2L.
\end{aligned}
\]
Consequently,
\begin{equation}\label{eq:shortconic}
\begin{aligned}
\sum_{\bx\in\mathcal A}\ell_{\bx}^{-3/4}
\ll
\sum_{\substack{k\ge0\\2^k\ll X}}
2^{-3k/4}
\#\{\bx\in\mathcal A:\ell_{\bx}\le2^{k+1}\}
\ll
X^2\sum_{\substack{k\ge0\\2^k\ll X}}2^{k/4}
\ll X^{9/4}.
\end{aligned}
\end{equation}

For $0\le i<j\le2$, write
$
d_{ij}=\gcd(|x_i|,|x_j|).$
Since $\bx$ is primitive, $d_{01},d_{02},d_{12}$ are pairwise
coprime, and
\[
g(\bx)=d_{01}d_{02}d_{12},
\qquad
d_{01}d_{02}\mid x_0,\quad
d_{01}d_{12}\mid x_1,\quad
d_{02}d_{12}\mid x_2.
\]
There are at most $2X/d$ nonzero multiples of $d$ in $[-X,X]$.
Therefore
\begin{equation}\label{eq:g-three-quarter-moment}
\begin{aligned}
\sum_{\bx\in\mathcal A}g(\bx)^{3/4}
\le
8X^3
\sum_{d_{01},d_{02},d_{12}\ge1}
\frac{(d_{01}d_{02}d_{12})^{3/4}}
     {(d_{01}d_{02}d_{12})^2}
=
8\zeta(5/4)^3X^3
\ll X^3.
\end{aligned}
\end{equation}
Combining \eqref{eq:shortconic} and
\eqref{eq:g-three-quarter-moment} by H\"older's inequality gives
\[
\begin{aligned}
\sum_{\bx\in\mathcal A}
\sqrt{g(\bx)}\,\ell_{\bx}^{-1/4}
\le
\left(\sum_{\bx\in\mathcal A}g(\bx)^{3/4}\right)^{2/3}
\left(\sum_{\bx\in\mathcal A}\ell_{\bx}^{-3/4}\right)^{1/3}
\ll
(X^3)^{2/3}(X^{9/4})^{1/3}
=
X^{11/4}.
\end{aligned}
\]
Since $Y_{\bx}\asymp_{c_-,c_+}Y$ and $\#\mathcal A\ll X^3$,
\eqref{eq:pointwise-error} implies
\[
\begin{aligned}
&\sum_{\bx\in\mathcal A}
\left|P_{\bx,w}(Y_{\bx})-\nu_{\bx,w}Y_{\bx}\right|\\
&\quad\ll_{a,w,\varepsilon,c_-,c_+}
X^{-1/2+\varepsilon}Y^{3/4}\log(2Y)
\sum_{\bx\in\mathcal A}
\sqrt{g(\bx)}\,\ell_{\bx}^{-1/4}
+X^{3+\varepsilon}\\
&\quad\ll_{a,w,\varepsilon,c_-,c_+}
X^{9/4+\varepsilon}Y^{3/4}\log(2Y)
+X^{3+\varepsilon}\\
&\quad\ll_{a,w,\varepsilon,c_-,c_+}
X^{2+\varepsilon}Y
\left(\frac XY\right)^{1/4}\log(2Y),
\end{aligned}
\]
where the last step uses $Y\ge X$. Finally, if $c_-Y<2$, then
\[
2\le X\le Y<2/c_-,
\qquad
Y_{\bx}\le2c_+/c_-.
\]
The coefficient vectors and the relevant integer points then range
over finite sets of bounded size, so this remaining case is absorbed
into the implied constant.
\end{proof}

\subsection{The full height region}

\begin{proposition}\label{thm:global-error}
Let $R(B)$ be defined in Theorem~\ref{thm:split}.
For $T=B^{1/3}$, put
\[
\mathcal F(B)=\sum_{|\bx|\le T}
\kappa_{[\bx]}\left(\frac B{|\bx|^2}-T\right),
\]
where the sum runs over projective coefficient points with nonzero
coordinates, each counted once using a primitive integral representative.
Then, as $B\to\infty$,
\begin{equation}\label{eq:global-error}
R(B)-\mathcal F(B)=o(B\log B).
\end{equation}
\end{proposition}
\begin{proof}
All coefficient sums below run over projective points with nonzero
coordinates, using one primitive integral representative for each point.
Write
$Y_{\bx}=\frac{B}{|\bx|^2}.$
Then
\[
R(B)-\mathcal F(B)
=
\sum_{|\bx|\le T}
\left\{
C_{[\bx]}(Y_{\bx})-C_{[\bx]}(T)
-\kappa_{[\bx]}(Y_{\bx}-T)
\right\}.
\]
For $|\bx|\le T$, both differences are nonnegative:
\[
0\le C_{[\bx]}(Y_{\bx})-C_{[\bx]}(T)
\le C_{[\bx]}(Y_{\bx}),
\qquad
0\le\kappa_{[\bx]}(Y_{\bx}-T)
\le\kappa_{[\bx]}Y_{\bx}.
\]

Fix $0<a,\beta<\frac12$ and $0<\eta<\frac13.$
The finitely many fibers with $|\bx|=1$ contribute $O(B)$, by
\eqref{eq:conic-upper}. Proposition~\ref{lem:shape} and
\eqref{eq:kappa-slab}, with $\gamma=1/12$, give
\begin{equation}\label{eq:badshape-error}
\sum_{\substack{
1<|\bx|\le T\\
\min_i|x_i|\le a|\bx|}}
\left\{
C_{[\bx]}(Y_{\bx})+\kappa_{[\bx]}Y_{\bx}
\right\}
\ll Ba^{2/3}\log(2B)+B
\end{equation}
and
\begin{equation}\label{eq:boundary-error}
\begin{aligned}
\sum_{B^{1/3-\eta}<|\bx|\le T}
\left\{
C_{[\bx]}(Y_{\bx})+\kappa_{[\bx]}Y_{\bx}
\right\}
\ll
B\left(1+\log\frac{T}{B^{1/3-\eta}}\right)
\ll B+\eta B\log B.
\end{aligned}
\end{equation}
It remains to consider
\[
\mathcal G=
\left\{
[\bx]:
1<|\bx|\le B^{1/3-\eta},\
\min_i|x_i|>a|\bx|
\right\}.
\]

Choose fixed even nonnegative weights
$w_\pm\in C_c^\infty(\mathbb R^3)$ such that
\[
\mathbf1_{(1-\beta)[-1,1]^3}
\le w_-
\le\mathbf1_{[-1,1]^3}
\le w_+
\le\mathbf1_{(1+\beta)[-1,1]^3}.
\]
The scaling identity
\[
s\!\left(q_{\bx};\mathbf1_{\lambda[-1,1]^3}\right)
=\lambda s(\bx)
\qquad(\lambda>0)
\]
and \eqref{eq:kappa} imply
\[
(1-\beta)\kappa_{[\bx]}
\le\frac{\nu_{\bx,w_-}}2
\le\kappa_{[\bx]}
\le\frac{\nu_{\bx,w_+}}2
\le(1+\beta)\kappa_{[\bx]}.
\]
Each nonsingular conic has at most six projective points with a zero
coordinate. Hence, for every $Z>0$,
\[
\frac12P_{\bx,w_-}(Z)-6
\le C_{[\bx]}(Z)
\le\frac12P_{\bx,w_+}(Z),
\]
and therefore
\[
\left|C_{[\bx]}(Z)-\kappa_{[\bx]}Z\right|
\le
\frac12\sum_{\pm}
\left|P_{\bx,w_\pm}(Z)-\nu_{\bx,w_\pm}Z\right|
+\beta\kappa_{[\bx]}Z+6.
\]
Applying this at both endpoints gives
\[
\begin{aligned}
&\left|
C_{[\bx]}(Y_{\bx})-C_{[\bx]}(T)
-\kappa_{[\bx]}(Y_{\bx}-T)
\right|\\
&\qquad\le
\frac12\sum_{\pm}\sum_{Z\in\{Y_{\bx},T\}}
\left|P_{\bx,w_\pm}(Z)-\nu_{\bx,w_\pm}Z\right|
+\beta\kappa_{[\bx]}(Y_{\bx}+T)+12.
\end{aligned}
\]

Partition $\mathcal G$ into dyadic blocks
\[
X/2<|\bx|\le X,
\qquad
X=2^j,\quad j\ge1,\quad
X\le2B^{1/3-\eta}.
\]
For sufficiently large $B$,
\[
X<T,\qquad
\frac a2X<|x_i|\le X,
\qquad
\frac B{X^2}\le Y_{\bx}<\frac{4B}{X^2},
\qquad
\frac B{X^2}\ge T\ge X.
\]
Thus Proposition~\ref{lem:avg-error}, with balance parameter $a/2$,
applies to both endpoints. For $w\in\{w_-,w_+\}$,
\[
\begin{aligned}
\sum_{\substack{
[\bx]\in\mathcal G\\
X/2<|\bx|\le X}}
\sum_{Z\in\{Y_{\bx},T\}}
\left|P_{\bx,w}(Z)-\nu_{\bx,w}Z\right|
&\ll_{a,w,\varepsilon}
\left(
B^{3/4}X^{3/4+\varepsilon}
+T^{3/4}X^{9/4+\varepsilon}
\right)\log(2B)\\
&\ll_{a,w,\varepsilon}
B^{3/4}X^{3/4+\varepsilon}\log(2B).
\end{aligned}
\]
Taking $\varepsilon=\eta/4$ and summing the geometric dyadic series,
we obtain
\begin{equation}\label{eq:good-error}
\begin{aligned}
\sum_{[\bx]\in\mathcal G}
\sum_{\pm}\sum_{Z\in\{Y_{\bx},T\}}
\left|P_{\bx,w_\pm}(Z)-\nu_{\bx,w_\pm}Z\right|
&\ll_{a,\beta,\eta}
B^{3/4}
\left(B^{1/3-\eta}\right)^{3/4+\eta/4}
\log(2B)\\
&\ll_{a,\beta,\eta}
B^{1-2\eta/3}\log(2B)
=o(B).
\end{aligned}
\end{equation}
By \eqref{eq:Kupper},
\[
\begin{aligned}
\sum_{[\bx]\in\mathcal G}
\kappa_{[\bx]}(Y_{\bx}+T)
\le
B\sum_{|\bx|\le T}
\frac{\kappa_{[\bx]}}{|\bx|^2}
+TK(T)
\ll B\log(2T)+T^3
\ll B\log(2B).
\end{aligned}
\]
Moreover,
\[
\#\mathcal G\ll B^{1-3\eta}=o(B).
\]
Consequently, for fixed $a,\beta,\eta$,
\[
\limsup_{B\to\infty}
\frac{1}{B\log B}
\sum_{[\bx]\in\mathcal G}
\left|
C_{[\bx]}(Y_{\bx})-C_{[\bx]}(T)
-\kappa_{[\bx]}(Y_{\bx}-T)
\right|
\ll\beta.
\]
Combining this with \eqref{eq:badshape-error},
\eqref{eq:boundary-error}, and the height-one contribution gives
\[
\limsup_{B\to\infty}
\frac{|R(B)-\mathcal F(B)|}{B\log B}
\ll a^{2/3}+\eta+\beta,
\]
with an implied constant independent of $a,\eta,\beta$.
Letting $a,\eta,\beta\to0$ after taking the limit superior proves
\[
R(B)-\mathcal F(B)=o(B\log B).
\]
\end{proof}

\section{Abel summation and the main theorem}\label{sec:conclusion}

\begin{proposition}\label{thm:high}
Equation \eqref{eq:high-main} holds.
\end{proposition}
\begin{proof}
Put $T=B^{1/3}$.
If $|\by|>T$ and $|\bx|^2|\by|\le B$, then necessarily $|\bx|<T$.
Let $C_{[\mathbf x]}(Y)$ be defined in Lemma \ref{prop:fixed}.
Thus we have
\[
R(B)=\sum_{|\bx|\le T}
\left(C_{[\bx]}\!\left(\frac B{|\bx|^2}\right)-C_{[\bx]}(T)\right).
\]
If $T$ is an integer and $|\bx|=T$, the corresponding difference is zero,
so there is no boundary issue.
From Lemma \ref{prop:fixed} and Proposition \ref{thm:global-error}, we have
\[
R(B)=
B\sum_{|\bx|\le T}\frac{\kappa_{[\bx]}}{|\bx|^2}
-TK(T)+o(B\log B),
\]
where $K(T)$ is defined in Lemma \ref{cor:kappa-upper}. Abel partial summation, including the atom at height $1$, gives
\[
\sum_{|\bx|\le T}\frac{\kappa_{[\bx]}}{|\bx|^2}
=\frac{K(T)}{T^2}+2\int_1^T\frac{K(t)}{t^3}\,dt.
\]
Since $B/T^2=T$, we have
\[
R(B)=
2B\int_1^T\frac{K(t)}{t^3}\,dt+o(B\log B).
\]
By Proposition \ref{thm:K},
\[
\int_1^T\frac{K(t)}{t^3}\,dt
=\cP\log T+o(\log T).
\]
Since $T=B^{1/3}$, we have
\[
R(B)
=\frac{2\cP}{3}B\log B+o(B\log B).
\]
\end{proof}

\begin{proof}[Proof of Theorems \ref{thm:split} and \ref{thm:main}]
The low-$y$ formula follows from Proposition \ref{thm:low},
and the high-$y$ formula from Proposition \ref{thm:high}.
The two regions are disjoint and contain all the points, so
\[
N_U(B)=N_{\mathrm{low}}(B)+R(B)
=\cP B\log B+o(B\log B).
\]
Proposition~\ref{prop:peyre} identifies $\cP$ with Peyre's constant.
This completes the proof of Theorems~\ref{thm:main} and~\ref{thm:split}.
\end{proof}

\subsection*{Acknowledgements}

The author is grateful to Prof. Daniel Loughran for drawing his attention to the joint work \cite{HL2025} with D. R. Heath-Brown and for suggesting an effective-cone interpretation of the unequal contributions in Theorem \ref{thm:split}. This work is supported by the National Natural Science Foundation of China
(Grant Nos.\ 12571002, 12501007 and 12526407), Young Elite Scientists
Sponsorship Program by Henan Association for Science and Technology
(Grant No.\ 2026HYTP001), High-level Talent Research Start-up Project
Funding of Henan Academy of Sciences (Grant Nos.\ 241819033, 252019083 and
242019104) and High-level Achievement Award and Cultivation Project Funding
of Henan Academy of Sciences (Grant Nos.\ 20262319007, 20262319001 and
20262319003).

\subsection*{AI Declaration}
The Proposition \ref{prop:peyre} was written with the help of ChatGPT 6.0.
ChatGPT 6.0 is used to polish the presentation and check the errors.
We remain fully responsible for all parts of the manuscript, including those helped by ChatGPT 6.0.

\bigskip
\noindent
\begin{minipage}{\textwidth}
\footnotesize
\textsc{Institute of Mathematics, Henan Academy of Sciences,
Zhengzhou, Henan 450046, P.~R.~China}\par
\noindent\textit{E-mail address:}
{\hypersetup{urlcolor=black}%
\href{mailto:zhaoxiaodong@hnas.ac.cn}{\texttt{zhaoxiaodong@hnas.ac.cn}}}
\end{minipage}


\begin{thebibliography}{99}
\footnotesize
\setlength{\itemsep}{4pt}

\bibitem{BT96}
V.~V.~Batyrev and Y.~Tschinkel,
Rational points on some Fano cubic bundles,
\emph{C. R. Acad. Sci. Paris S\'er. I Math.}
\textbf{323} (1996), no.~1, 41--46.

\bibitem {B6} V. Blomer and J. Br\"{u}dern, Counting in hyperbolic spikes: the diophantine analysis of multihomogeneous diagonal equations, \emph{J. Reine Angew. Math.} \textbf{737} (2018), 255--300.

\bibitem{BBH2024}
D.~Bonolis, T.~Browning and Z.~Huang,
Density of rational points on some quadric bundle threefolds,
\emph{Math. Ann.} \textbf{390} (2024), 4123--4207.

\bibitem{BHB}
T.~D.~Browning and D.~R.~Heath-Brown,
Counting rational points on hypersurfaces,
\emph{J. Reine Angew. Math.} \textbf{584} (2005), 83--115.

\bibitem{BH2020}
T.~D.~Browning and D.~R.~Heath-Brown,
Density of rational points on a quadric bundle in
$\mathbb P^3\times\mathbb P^3$,
\emph{Duke Math. J.} \textbf{169} (2020), no.~16, 3099--3165.

\bibitem {B5} T. D. Browning and L. Q. Hu, Counting rational points on biquadratic hypersurfaces, \emph{Adv. Math.} \textbf{349} (2019), 920--940.

\bibitem{BMW}
T.~Browning, R.~Munshi and V.~Y.~Wang,
Beyond the square-root barrier: cubic forms of Perazzo type,
arXiv:2604.19045v1, 2026.

\bibitem{Cassels1997}
J.~W.~S. Cassels,
\emph{An Introduction to the Geometry of Numbers},
Classics in Mathematics, Springer-Verlag, Berlin, 1997,
reprint of the 1971 edition.

\bibitem{FMT}
J.~Franke, Y.~I.~Manin and Y.~Tschinkel,
Rational points of bounded height on Fano varieties,
\emph{Invent. Math.} \textbf{95} (1989), no.~2, 421--435.

\bibitem{FriedlanderIwaniec2010}
J. Friedlander and H. Iwaniec,
Ternary quadratic forms with rational zeros,
\textit{J. Th\'eor. Nombres Bordeaux}
\textbf{22} (2010), no.~1, 97--113.

\bibitem{Hanke}
J.~Hanke,
Local densities and explicit bounds for representability
by a quadratic form,
\emph{Duke Math. J.} \textbf{124} (2004), no.~2, 351--388.

\bibitem{Hartshorne1977}
R. Hartshorne,
\emph{Algebraic Geometry}, Graduate Texts in Mathematics,
vol.~52, Springer-Verlag, New York--Heidelberg, 1977.

\bibitem{HB}
D. R.~Heath-Brown,
The distribution of rational points on conics,
\emph{Acta Arith.} \textbf{209} (2023), 91--128.

\bibitem{HL2025}
D.~R.~Heath-Brown and D.~Loughran,
Manin's conjecture for quintic del Pezzo surfaces
with a conic bundle structure, arXiv:2506.02829, 2025.

\bibitem{LB}
P.~Le Boudec,
Density of rational points on a certain smooth bihomogeneous threefold,
\emph{Int. Math. Res. Not. IMRN} (2015), no.~21, 10703--10715.

\bibitem{LST}
B.~Lehmann, A.~K.~Sengupta and S.~Tanimoto,
Geometric consistency of Manin's conjecture,
\emph{Compos. Math.} \textbf{158} (2022), no.~6, 1375--1427.

\bibitem{LN}
R.~Lidl and H.~Niederreiter,
\emph{Finite Fields},
2nd ed., Encyclopedia of Mathematics and its Applications,
vol.~20, Cambridge University Press, Cambridge, 1997.

\bibitem{Peyre}
E.~Peyre,
Hauteurs et mesures de Tamagawa sur les vari\'et\'es de Fano,
\emph{Duke Math. J.} \textbf{79} (1995), no.~1, 101--218.

\bibitem{Poonen2017}
B. Poonen,
\emph{Rational Points on Varieties}, Graduate Studies in Mathematics,
vol.~186, American Mathematical Society, Providence, RI, 2017.

\bibitem{Schindler}
D.~Schindler,
Manin's conjecture for certain biprojective hypersurfaces,
\emph{J. Reine Angew. Math.} \textbf{714} (2016), 209--250.


\bibitem{Serre}
J.-P.~Serre,
\emph{A Course in Arithmetic},
Graduate Texts in Mathematics, vol.~7, Springer, 1973.

\bibitem{Sofos}
E.~Sofos,
Uniformly counting rational points on conics,
\emph{Acta Arith.} \textbf{166} (2014), no.~1, 1--13.

\bibitem{Toth2014}
L.~T\'oth,
Multiplicative arithmetic functions of several variables:
a survey,
in \emph{Mathematics Without Boundaries:
Surveys in Pure Mathematics},
T.~M.~Rassias and P.~M.~Pardalos (eds.),
Springer, New York, 2014, 483--514.

\bibitem{W1}
X.~Wang,
Density of rational points on diagonal bidegree (1,2) hypersurface in $\mathbb P^{s-1}\times\mathbb P^{s-1}$,
\emph{Funct. Approx. Comment. Math.} \textbf{70} (2024), no.~2, 201--231.

\bibitem{Widmer2012} M.~Widmer, Lipschitz class, narrow class, and counting lattice points, \emph{Proc. Amer. Math. Soc.} \textbf{140} (2012), no.~2, 677--689.

\bibitem{Z1}
X. D. Zhao,
Density of rational points on a certain biprojective hypersurface,
\emph{Acta Arith.} \textbf{219} (2025), no.~1, 1--31.

\bibitem {Z2} X. D. Zhao, The Manin--Peyre conjecture for certain multiprojective hypersurfaces, \emph{Int. J. Number Theory} \textbf{20} (2024), 1265--1285.

\end{thebibliography}
\end{document}